\documentclass[11pt, a4paper, twoside]{amsart}
\usepackage[utf8]{inputenc}
\usepackage{amsmath}
\usepackage{amsthm}
\usepackage{amssymb}
\usepackage{array}
\usepackage{arydshln}
\usepackage{extpfeil} 
\DeclareMathOperator{\add}{add}
\DeclareMathOperator{\Hom}{Hom}
\DeclareMathOperator{\End}{End}
\DeclareMathOperator{\Ext}{Ext}
\DeclareMathOperator{\modf}{mod}

\newtheorem{definition}{Definition}
\newtheorem{theorem}[definition]{Theorem}
\newtheorem{lemma}[definition]{Lemma}

\newtheorem{corollary}[definition]{Corollary}

\usepackage{tikz}

\usetikzlibrary{arrows,decorations.pathmorphing,decorations.pathreplacing, decorations.markings} 
\usepackage{caption}

\tikzset{vertex/.style={circle,fill=black,inner sep=1pt,outer sep=2pt},
         mvertex/.style={rectangle,draw=black,thick,inner sep=2pt,outer sep=2pt},
         tvertex/.style={inner sep=1pt,font=\scriptsize},
         unvertex/.style={circle,fill=white,draw=white,inner sep=1pt},
         bvertex/.style={circle,fill=white,draw=white,inner sep=1pt, font=\scriptsize},
         fill0/.style={fill=black!0,draw=black!0},
         fill1/.style={fill=black!15,draw=black!15},
         fill2/.style={fill=black!30,draw=black!30},
         fill12/.style={fill=black!45,draw=black!45},
         fill3/.style={fill=black!60,draw=black!60},
         >=stealth',
         leadsto/.style={-angle 90,decorate,decoration=snake,very thick},
         cut/.style={decorate,decoration=saw,very thick}}

\newcommand{\replacevertex}[3][fill=white,draw=white]
 {
  \node at #2 [#1,circle,inner sep=1pt] {};
  \node #2 at #2 #3;
 }

\makeatletter
\providecommand*{\twoheadrightarrowfill@}{%
  \arrowfill@\relbar\relbar\twoheadrightarrow
  }
\providecommand*{\xtwoheadrightarrow}[2][]{%
  \ext@arrow 0579\twoheadrightarrowfill@{#1}{#2}%
  }
\makeatother

\begin{document}
\title[Periodicity of cluster tilting objects]{Periodicity of cluster tilting objects}

\author{Benedikte Grimeland}

\begin{abstract}
Let \(\mathcal T\) be a locally finite triangulated category with an autoequivalence \(F\) such that the orbit category \(\mathcal T /F\) is triangulated. We show that if \(\mathcal X\) is an \(m\)-cluster tilting subcategory, then the image of \(\mathcal X\) in  \(\mathcal T/F\) is an \(m\)-cluster tilting subcategory if and only if \(\mathcal X\) is \(F\)-perodic. 

We show that for path-algebras of Dynking quivers \(\Delta\) one may study the periodic properties of \(n\)-cluster tilting objects in the \(n\)-cluster category \(\mathcal C_n(k\Delta)\) to obtain information on periodicity of the preimage as \(n\)-cluster tilting subcategories of \(\mathcal D^b(k\Delta)\). 

Finally we classify the periodic properties of all \(2\)-cluster tilting objects \(T\) of Dynkin quivers, in terms of symmetric properties of the quivers of the corresponding cluster tilted algebras \(\End_{\mathcal C_2}(T)^{op}\). This gives a complete overview of all \(2\)-cluster tilting objects of all orbit categories of Dynkin diagrams. 
\end{abstract}
\maketitle

\section{Introduction}
Cluster categories were introduced as a categorification of the combinatorics in cluster algebras, which were introduced in \cite{FZ}. This was done for path algebras of Dynkin diagrams of type A by \cite{ccs}, and more general for finite-dimensional hereditary algebras by \cite{bmrrt}. With this category in place, the authors of \cite{bmrrt} were able to generalize the notion of tilting to the notion of cluster tilting. In the cluster category one has cluster tilting objects, which give rise to cluster tilted algebras. The notion of a cluster tilting object was however generalized again, see \cite{Iyama1} and \cite{Iyama2}, to the notion of a cluster tilting subcategory of a triangulated category, and \(n\)-cluster tilting subcategory of a triangulated category. 

Among the \(n\)-cluster tilting subcategories, the case \(n=2\) is of special interest, as results by \cite{bmr} 
and \cite{k-zhu} show that quotients of triangulated categories by a \(2\)-cluster tilting subcategory produces an abelian category with enough projectives. However \(n\)-cluster tilting subcategories in general are also important in connection with higher dimensional Auslander-Reiten theory (\cite{Iyama1},\cite{Iyama2}).

In this paper we study how periodicity of \(n\)-cluster tilting subcategories under certain functors in a triangulated category gives rise to \(n\)-cluster tilting subcategories in orbit categories of the triangulated category. In particular we show that for representation-finite hereditary algebras \(H\), the periodicity of an \(n\)-cluster tilting object of the \(m\)-cluster tilting category \(\mathcal C_m(H)\) under the suspension functor and the Auslander-Reiten translate carries over to the preimage \(\mathcal X\) of the bounded derived category \(\mathcal D^b(H)\). As there is a one-to-one correspondence between \(m\)-cluster tilting subcategories of \(\mathcal D^b(H)\) and \(m\)-cluster tilting objects of \(\mathcal C_m(H)\), this enables us to study the periodic properties within the cluster categories and still obtain the necessary information about periodicity in the bounded derived category \(\mathcal D^b(H)\). 

The mutation classes of quivers of cluster tilted algebras of representation-finite hereditary algebras have been described in \cite{DagfinnA}, \cite{DagfinnD}, and \cite{typeE}. Based on the symmetry properties of the quivers of the corresponding cluster tilted algebras, we classify the periodic properties of all \(2\)-cluster tilting objects of representation-finite hereditary algebras as an application of the method described above. This gives a complete overview of all \(2\)-cluster tilting objects in all triangulated orbit categories of representation-finite hereditary algebras, with an easy way to determine all the indecomposable summands of the object. In regard to this, it is interesting to note two things. First, in many cases the combinatorics of type D is more involved than the combinatorics of type A. However this is not true for this case, the combinatorics of the proof for the 3-symmetric cases of type A is conceptuallly not any easier than the proofs of type D. Second, when considering cases of Dynkin type \(D_n\), the parity of \(n\) affects the results in a substantial way.

Earlier results on the classification of orbit categories with a \(2\)-cluster tilting object are given in \cite{bikr}, \cite{bpr} and for a special case for type E in \cite{ladkani}.  These results covers \(2\)-cluster tilting objects in \(2\)-Calabi Yau triangulated orbit categories, and the method used is to study orbit-categories of the \(2\)-cluster category. As described above, our method can be applied regardless of the Calabi-Yau dimension of the category. Moreover, our method in itself provide a combinatorial means for finding explicit descriptions of any \(2\)-cluster tilting object \(T\) in any orbit category in terms of its indecomposable summands, and hence also the quiver of the corresponding cluster tilted algebra. Furthermore, as noticed above, our method may be applied to any known \(m\)-cluster tilting object \((m>2)\) of a triangulated orbit category,  to find other orbit categories with \(m\)-cluster tilting objects arising from the same preimage in the bounded derived category.

The paper is organized as follows: Section \ref{background} contains the basic definitions from cluster tilting theory that will be used throughout the paper. In section \ref{theory} we show how periodicity of cluster tilting subcategories in a triangulated category \(\mathcal T\) determines cluster tilting subcategories in triangulated orbit categories of \(\mathcal T\). Using properties of the bounded derived category for representation-finite hereditary algebras we show that some properties of periodic cluster tilting subcategories of cluster categories carry over to the bounded derived category. Based on this we describe our main method, which is applied in later sections.

Section \ref{quivers} give a short overview of the quivers of cluster tilted algebras of Dynkin type \(A\) and \(D\), and also states results connected to quivers of cluster tilted algebras, that will be used in sections \ref{sectA} and \ref{typeD}. Then in section \ref{sectA} we use the method described in section \ref{theory}, and classify for all \(2\)-cluster tilting subcategories of \(\mathcal D^b(A_n)\) which functors they are periodic under. The classification is given in terms of symmetric properties of the quivers of the corresponding \(2\)-cluster tilting subcategories in the cluster category \(\mathcal C_2(A_n)\). In section \ref{typeD} we give the corresponding results for \(2\)-cluster tilting subcategories of type D. Type \(E\) is studied in section \ref{sectE}. Finally we consider the Euclidean and wild cases in section \ref{wildEuclid}.

\section{Background}\label{background}

For more background on the following definitions see for example \cite{Iyama}.
\begin{definition}
Let $\mathcal{C}$ be a category and $\mathcal{T}$ a full subcategory of $\mathcal{C}$. Then $\mathcal{T}$ is called a functorially finite subcategory if it is both covariantly finite and contravariantly finite. That is for each object $X\in\mathcal{C}$ there are objects $T,T'\in\mathcal{T}$ with morphisms  and $f:X\rightarrow T'$ and $g:T\rightarrow X$ such that for any object $T''\in \mathcal{T}$ there are epimorphisms $\Hom_{\mathcal{C}}(T',T'') \xtwoheadrightarrow{\Hom_{\mathcal{C}}(f,T'')} \Hom_{\mathcal{C}}(X,T'')$ and $\Hom_{\mathcal{C}}(T'',T) \xtwoheadrightarrow{\Hom_{\mathcal{C}}(T'',g)} \Hom_{\mathcal{C}}(T'',X)$.
\end{definition}

\begin{definition}
A functorially finite subcateogory $\mathcal{T}$ is called an $r$-cluster tilting subcategory of $\mathcal{C}$ if 
\begin{align*}
\mathcal{T}=&\left\{X\in\mathcal{C}|\Ext^i_{\mathcal{C}}(\mathcal{T},X)=0 \text{\ for any\ }0<i<r \right\}\\
						=&\left\{X\in\mathcal{C}|\Ext^i_{\mathcal{C}}(X,\mathcal{T})=0 \text{\ for any\ }0<i<r \right\}\\
\end{align*}
If $\mathcal{T}=\add T$ for some object $T$, then $T$ is called an $r$-cluster tilting object. 
\end{definition}

\begin{definition}\label{locallyFinite}
We call a triangulated category locally finite if for an indecomposable object \(X\) there is only a finite number of isomorphism classes of indecomposable objects \(Y\) such that \(\Hom_{\mathcal T}(X,Y)\neq0\). 
\end{definition}

Note that definition \ref{locallyFinite} implies its own dual (see \cite{ZhuXiao},\cite{amiot}), i.e. for any indecomposable object \(Y\) there are only a finite number of isomorphism classes of indecomposable objects \(X\) such that \(\Hom_{\mathcal T}(X,Y)\neq0\). Also note that any subcategory of a locally finite triangulated category is both contravariantly finite and covariantly finite.

Next we will define what we mean by an orbit category. One can define orbit categories of all additive categories (see \cite{keller}), however we will only use the definition for triangulated categories and therefore give the definition in this context. See for example also \cite{bmrrt}.

\begin{definition}
Let \(\mathcal T\) be a triangulated cateogory with an autoequivalence \(F:\mathcal T\rightarrow \mathcal T\). The orbit cateogory \(\mathcal O_F(\mathcal T)\) has the same objects as \(\mathcal T\). The morphisms \(\Hom_{\mathcal O}(X,Y)\) between two objects in \(\mathcal O_F(\mathcal T)\) are in bijection with the set \(\oplus \Hom_{\mathcal T}(X,F^nY)\). 
\end{definition}

In later sections we will mostly consider orbit categories of the bounded derived category \(\mathcal D^b(H)\) where \(H\) is a representation-finite hereditary algebra. We define the \(m\)-cluster category of \(H\) to be \(\mathcal C_m(H):=\mathcal D^b(H)/\tau^-\left[m-1\right]\) for \(m\geq 2\). In particular it is known (\cite{keller},\cite{bmrrt}) that the \(m\)-cluster category of an hereditary algebra is triangulated, and \(\pi_m:\mathcal D^b(H)\rightarrow  \mathcal C_m(H)\) is a triangle functor. 
Note that in a triangulated category with only finitely many isomorphism classes of indecomposable objects, a cluster tilting subcategory always gives rise to a cluster tilting object. We will often use the term cluster tilting object about a cluster tiling subcategory in such a setting. Furthermore, we note that the suspension functor in the bounded derived category of a representation-finite hereditary algebra \(\mathcal D^b(H)\) will be denoted \(\left[1\right]\), whereas the suspension functor in a general triangulated category will be denoted by \(\Sigma\). Throughout the paper \(\tau\) always refers to the AR-translate of the category in question. 

\section{Periodicity determines cluster tilting subcategories}\label{theory}

In this section we will study how periodicity under a functor $F$, will determine for an $m$-cluster tilting subcategory of $\mathcal D^b(H)$ in which orbit categories $\mathcal O_F$ the image will again be an $m$-cluster tilting subcategory. The first three results we are able to state in a somewhat more general setting, that is, for a locally finite triangulated category \(\mathcal T\) in general, rather than for just \(\mathcal D^b(H)\). For the remaining results however we will need some properties that are particular for \(\mathcal D^b(H)\).

These results will be the foundation of the method applied throughout the rest of the paper.

Starting out is a lemma showing that the preimage of an $m$-cluster tilting subcategory of an orbit category $\mathcal O_F$ is an $m$-cluster tilting subcategory of $\mathcal T$: 

\begin{lemma}\label{PreimCltilt}
Let $\mathcal T$ be a locally finite triangulated category, and $F$ an autoequivalence such that the orbit category $\mathcal O_F:=\mathcal T/F$ is triangulated and \(\pi_F:\mathcal T\rightarrow \mathcal O_F\) is a triangle functor. If $\mathcal Y$ is an $m$-cluster tilting subcategory of $\mathcal O_F$ then $\mathcal X:=\pi^{-1}_F(\mathcal Y)$ is an $m$-cluster tilting subcategory of $\mathcal T$.
\end{lemma}

\begin{proof}
Since $\mathcal T$ is locally finite the subcategory $\mathcal X$ is functorially finite. Also since the functor $\pi_F:\mathcal T \rightarrow \mathcal O_F$ is a faithful triangle functor we have the following inclusions:
$$\pi^{-1}_F(\mathcal Y)\subset\left\{Z\in\mathcal T|\Ext^i_\mathcal T(\pi^{-1}_F(\mathcal Y),Z)=0 \text{ for } 0<i<m\right\}$$

$$\pi^{-1}_F(\mathcal Y)\subset\left\{Z\in\mathcal T|\Ext^i_\mathcal T(Z,\pi^{-1}_F(\mathcal Y))=0 \text{ for } 0<i<m\right\}$$
 
We need to show that in fact we have equality, not just inclusion. Let \(Z\in\mathcal T\) be an object such that \(\Ext_{\mathcal T}^i(\pi_F^{-1}(\mathcal Y),Z)=0\) for \(0<i<m\), i.e. \(\Hom_{\mathcal T}(\pi_F^{-1}(\mathcal Y),\Sigma^iZ)=0\) for \(0<i<m\).
We want to show that \(Z\in\pi^{-1}_F(\mathcal Y)\), i.e. that \(\pi_F(Z)\in\mathcal Y\). 

Since \(\mathcal Y\) is an \(m\)-cluster tilting subcategory of \(\mathcal O_F\) we have that \[\mathcal Y=\left\{?\in\mathcal O_F|\Ext^i_{\mathcal O_F}(\mathcal Y,?)=0 \text{\ for } 0<i<m\right\}.\]

Let \(\overline Y\) be an object of \(\mathcal Y\) and let \(Y\) be its preimage in \(\pi^{-1}_F(\mathcal Y). \)
Then we have the following:
\begin{align*}
\Ext^i_{\mathcal O_F}(\overline Y,\pi_F(Z))&=\Hom_{\mathcal O_F}(\overline Y, \Sigma^i\pi_F(Z))\\
&=\Hom_{\mathcal O_F}(\pi_F(Y),\pi_F(\Sigma^iZ))\\
&=\bigoplus \Hom_{\mathcal T}(F^p(Y),\Sigma^i Z)=0
\end{align*}
for \(0<i<m\). Hence we  conclude that \(\pi_F(Z)\in\mathcal Y\).
The last inclusion can be shown in a similar way.
\end{proof}


The next result states a very important fact, if an $m$-cluster tilting category of a triangulated category $\mathcal T$ is periodic under a functor $F$, then the image in the orbit category $\mathcal O_F$ is an $m$-cluster tilting subcategory. 

\begin{lemma}\label{ImCltilt}
Let $\mathcal T$ be a locally finite triangulated category and $F$ an autoequivalence of $\mathcal T$ such that $\mathcal O_F:=\mathcal T/F$ is a triangulated category and \(\pi_F:\mathcal T\rightarrow \mathcal O_F\) is a triangle functor.  If $\mathcal X$ is an $F$-periodic $m$-cluster tilting subcategory of $\mathcal T$, then $\mathcal Y:=\pi_F(\mathcal X)$ is an $m$-cluster tilting subcategory of $\mathcal O_F$.
\end{lemma}

\begin{proof}
First let $\pi_F(X_1)$ and $\pi_F(X_2)$ be two objects of $\pi_F(\mathcal X)=\mathcal Y$. Then we have:
\begin{center}
\begin{align*}
\Hom_{\mathcal O_F}(\pi_F(X_1),\Sigma^i\pi_F(X_2))&=\Hom_{\mathcal O_F}(\pi_F(X_1),\pi_F(\Sigma^iX_2))\\
&=\oplus_{p\in\mathbb Z}\Hom_\mathcal T(X_1,\Sigma^iF^p(X_2))=0
\end{align*}
\end{center}
for \(0<i<m\). Where the last expression is zero since $F^p(X_2)\in F^p\mathcal X=\mathcal X$. This gives the inclusions:
$$\pi_F(\mathcal X)\subset\left\{\overline Z \in \mathcal O_F|\Ext^i_{\mathcal O_F}(\pi_F(\mathcal X),\overline Z)=0 \text{ for } 0<i<m\right\}$$
$$\pi_F(\mathcal X)\subset\left\{\overline Z\in\mathcal O_F|\Ext^i_{\mathcal O_F}(\overline Z,\pi_F(\mathcal X))=0\text{ for } 0<i<m\right\}.$$

For the remaining inclusions, start with an object $\overline Z$ of $\mathcal O_F$ such that $\Ext^i_{\mathcal O_F}(\overline Z,\pi_F(\mathcal X))=0$ for $0<i<m$, i.e. $\Hom_{\mathcal O_F}(\overline Z,\Sigma^i\pi_F(\mathcal X))=0$ for $0<i<m$. Let \(X\) be any object of \(\mathcal X\) and \(\overline X:=\pi_F(X)\). Then we have 
\begin{center}
\begin{align*}
0=\Hom_{\mathcal O_F}(\overline Z,\Sigma^i\overline X)&=\Hom_{\mathcal O_F}(\pi_F(Z),\Sigma^i\pi_F(X))\\
&=\Hom_{\mathcal O_F}(\pi_F(Z),\pi_F(\Sigma^i X))\\
&\supseteq \Hom_\mathcal T(Z,\Sigma^i X)\\
\end{align*}
\end{center}
Therefore $Z\in\mathcal X=\pi^{-1}(\mathcal Y)$ and $\overline Z=\pi_F(Z)\in\mathcal Y$. The last inclusion can be shown in a similar way. Functorial finiteness of \(\mathcal Y\) follows from \(\mathcal T\) being locally finite.
\end{proof}

\begin{theorem}\label{correspondence}
Let \(\mathcal T\) be a locally finite triangulated category and \(F\) an autoequivalence such that \(\mathcal O_F:=\mathcal T/F\) is triangulated and the projection functor  \(\pi_F:\mathcal T\rightarrow \mathcal O_F\) is a triangle functor. Then there is a bijection between the set of \(F\)-periodic \(m\)-cluster tilting subcategories of \(\mathcal T\) and \(m\)-cluster tilting subcategories of \(\mathcal O_F\).
\end{theorem}

\begin{proof}
This follows directly from lemma \ref{PreimCltilt} and lemma \ref{ImCltilt}.
\end{proof} 

For the rest of the section we will focus on the case when \(\mathcal T=\mathcal D^b(H)\) where \(H\) is a representation-finite hereditary algebra. Note that \(\mathcal D^b(H)\) is locally finite for any hereditary algebra \(H\), therefore it follows immediatly that all subcategories are both covariantly finite and contravariantly finite. Let \(\mathbb S _m\) denote the endofunctor \(\tau\left[1-m\right]\) on \(\mathcal D ^b(H)\). It is known that any $m$-cluster tilting subcategory is $\mathbb S_m$ periodic:

\begin{lemma}\label{periodicity1}\cite{bergh}
Let \(\mathcal T\) be an \(m\)-cluster tilting subcategory of \(\mathcal D^b(H)\) for a representation-finite hereditary algebra \(H\) and \(m\geq1\). Then \(\mathcal T\) is \(\mathbb S_m\) periodic, that is \(\mathbb S_m \mathcal T=\mathcal T\).
\end{lemma}

\begin{proof}
Let \(T_1\in\mathcal T\). Then for any \(T_2\in\mathcal T\) and 
for \(m>1\) we have that \(\Hom_{\mathcal D^b}(T_1,T_2\left[m-i\right] )=0\) for \(0<i<m\). Using Serre duality we then obtain:
\begin{align*}
0=\Hom_{\mathcal D^b}(T_1,T_2\left[m-i\right]) &\cong  D \Hom_{\mathcal D^b}(T_2\left[m-i\right],\mathbb S T_1) \\ & \cong D \Hom_{\mathcal D^b}(T_2,(\mathbb S T_1\left[-m\right])\left[i\right])\\ &\cong D \Hom_{\mathcal D^b}(T_2,(\mathbb S_m T_1)\left[i\right])
\end{align*}
Hence the object \(\mathbb S T_1\left[m\right]=\mathbb S_m T_1\) is in \(\add \mathcal T\). Similarly one can show that the object \(\mathbb S ^{-1}T_1\left[m\right]=\mathbb S_m^{-1}T_1\) also belongs to \(\mathcal T\).
\end{proof}

This implies that we have  a \(1-1\)-correspondence between \(m\)-cluster tilting subcategories in \(\mathcal D^b(H)\) and \(m\)-cluster tilting subcategories of the \(m\)-cluster category:

\begin{corollary}\label{correspondenceDbOgCm}
Let \(H\) be a representation-finite hereditary algebra. Then there is a 1-1 correspondence between \(m\)-cluster tilting subcategories of \(\mathcal D ^b(H)\) and \(m\)-cluster tilting subcategories of \(\mathcal C _m(H)=\mathcal D ^b(H)/\tau^-\left[m-1\right]\).
\end{corollary}

\begin{lemma}\label{periodInDerived}
Let \(\mathcal T\) be an \(n\)-cluster tilting subcategory of \(\mathcal C_m(H)\), and let \(\mathcal X\) be the preimage of \(\mathcal T\) in \(\mathcal D^b(H)\). If \(F_{\mathcal C}:\mathcal C_m(H)\rightarrow \mathcal C_m(H)\) is a functor of the form \(\tau^s\left[t\right]\), and \(\mathcal T\) is \(F_{\mathcal C}\)-periodic then \(\mathcal X\) is periodic under the functor \(F_{\mathcal D^b}:\mathcal D^b(H)\rightarrow \mathcal D^b(H)\) where \(F_{\mathcal D^b}=\tau^s\left[t\right]\) .
\end{lemma}

\begin{proof}
It is known from \cite{bmrrt} and \cite{keller} that the \(m\)-cluster category is triangulated and that \(\pi_m:\mathcal D^b(H)\rightarrow \mathcal C_m(H)\) is a triangle functor for \(m\geq 2\). Since \(\pi_m\) is a triangle functor it commutes with the suspension functors in \(\mathcal D^b(H)\) and \(\mathcal C_m(H)\) (both denoted \(\left[1\right]\) by abuse of notation).

From proposition 1.3 in \cite{bmrrt} we also have that \(\pi_m\) commutes with the AR-translate in both categories. 

Let \(X\) be any object in \(\mathcal X\). We need to show that \(F_{\mathcal D}X\in\mathcal X\). We therefore have that \(F_{\mathcal C}(\pi_m(X))=\pi_m(F_{\mathcal D}(X))\), and hence since \(F_{\mathcal C}(\pi_m(X))\in\mathcal T\) we have that \(F_{\mathcal D}(X)\in\mathcal X\). 
\end{proof}

We will now describe the method that will be applied in sections \ref{sectA}, \ref{typeD} and \ref{sectE}, where we will continue the focus on representation-finite hereditary algebras.  As we see from theorem \ref{correspondence}, whether or not an \(m\)-cluster tilting subcategory \(\mathcal X\) of  \(\mathcal D^b(H)\) can be pushed down to an \(m\)-cluster tilting subcategory in an orbit category \(\mathcal O_F(H)\), depends on if \(\mathcal X\) is periodic under the functor \(F\) in \(\mathcal D^b(H)\). Showing that a subcategory is \(F\)-periodic for a certain functor \(F\) will in the cases we will study require some counting, and is most easily carried out in a category with only finitely many isomorphism classes of indecomposable objects (i.e. in \(\mathcal C_m(H)\) rather than in in \(\mathcal D^b(H))\). However, by corollary \ref{correspondenceDbOgCm} and lemma \ref{periodInDerived}, for an \(m\)-cluster tilting subcategory of \(\mathcal D^b(H)\) with image \(\pi_m(\mathcal X)\) in the \(m\)-cluster category, we may study the periodicity of \(\pi_m(\mathcal X)\) under functors of the form \(\tau^s\left[t\right]\) in the \(m\)-cluster category to obtain information about the periodicity of \(\mathcal X\) under \(\tau^s\left[t\right]\) in \(\mathcal D^b(H)\). 

We will use this method to determine the period-properties under functors of the form \(\tau^s\left[t\right]\) for all \(2\)-cluster tilting subcategories of \(\mathcal D^b(H)\) for \(\Delta\) a Dynkin diagram. Note that in \(2\)-cluster categories \(\mathcal C_2(H)\), we have the isomorphism \(\tau^-\left[1\right]\cong id\) as functors, and hence \(\tau\cong\left[1\right]\). Hence we only need to check for periodicity under one of the functors \(\tau^s\) or \(\left[t\right]\). In sections \ref{sectA}, \ref{typeD} and \ref{sectE} we will study the periodicity of all \(2\)-cluster tilting subcategories of the \(2\)-cluster category of respectively type \(A,D\) and \(E\) under functors of the type \(\tau^s\). Finally we note that in these cases all \(2\)-cluster tilting subcategories are also \(2\)-cluster tilting objects, as the \(2\)-cluster category in these cases have only finitely many isomorphism classes of indecomposable objects.

\section{Quivers of cluster tilted Algebras}\label{quivers}
For each cluster tilting object $T$ in a cluster category $\mathcal{C}_2(kQ)$, with \(Q\) a Dynkin diagram, there is also a cluster tilted algebra defined as $\End_{\mathcal{C}}(T)^{op}$.  In subsequent sections we will investigate some properties of \(2\)-cluster tilting objects of type \(A\) and \(D\) based on the shape of the quiver of the corresponding cluster tilted algebra. The possible quivers occuring as quivers of cluster tilted algebras of type $A_n$ were classified by Vatne \cite{DagfinnA}. The first classification of quivers of cluster tilted algebras of type $D_n$ was given in \cite{DagfinnD} also by Vatne and contained four main types. In \cite{bow}, Oppermann, Bertani-{\O}kland and Wr{\aa}lsen were able to reduce this description to three main types of quivers. We will refer to the classification as it is presented in \cite{bow}.

\subsection{Cluster tilted algebras of type A} 

All cluster tilted algebras of type $A$ are a connected subquiver of the quiver $\mathcal{Q}$ below

\begin{minipage}{\linewidth}\label{Q_A}
\[ \begin{tikzpicture}[xscale=.78,yscale=.4,yscale=-1]
 \node (P1) at (13.6,2) [vertex] {};
 \node (P2) at (11.6,4) [vertex] {};
 \node (P3) at (15.6,4) [vertex] {};
 \node (P4) at (10.6,5.5) [vertex] {};
 \node (P5) at (12.6,5.5) [vertex] {};
 \node (P6) at (14.6,5.5) [vertex] {};
 \node (P7) at (16.6,5.5) [vertex] {};
 \node (P8) at (10.1,6.5) [vertex] {};
 \node (P9) at (11.1,6.5) [vertex] {};
 \node (P10) at (12.1,6.5) [vertex] {};
 \node (P11) at (13.1,6.5) [vertex] {};
 \node (P12) at (14.1,6.5) [vertex] {};
 \node (P13) at (15.1,6.5) [vertex] {};
 \node (P14) at (16.1,6.5) [vertex] {};
 \node (P15) at (17.1,6.5) [vertex] {};
 \draw [->] (P2) -- (P1);
 \draw [->] (P1) -- (P3);
 \draw [->] (P3) -- (P2);
 \draw [->] (P4) -- (P2);
 \draw [->] (P2) -- (P5);
 \draw [->] (P6) -- (P3);
 \draw [->] (P3) -- (P7);
 \draw [->] (P5) -- (P4);
 \draw [->] (P7) -- (P6);
 \draw [->] (P8) -- (P4);
 \draw [->] (P4) -- (P9);
 \draw [->] (P10) -- (P5);
 \draw [->] (P5) -- (P11);
 \draw [->] (P12) -- (P6);
 \draw [->] (P6) -- (P13);
 \draw [->] (P14) -- (P7);
 \draw [->] (P7) -- (P15);
 \draw [->] (P15) -- (P14);
 \draw [->] (P13) -- (P12);
 \draw [->] (P11) -- (P10);
 \draw [->] (P9) -- (P8);
 \draw [dotted] (P8) -- +(-.3,.7);
 \draw [dotted] (P9) -- +(.3,.7);
 \draw [dotted] (P10) -- +(-.3,.7);
 \draw [dotted] (P11) -- +(.3,.7);
 \draw [dotted] (P12) -- +(-.3,.7);
 \draw [dotted] (P13) -- +(.3,.7);
 \draw [dotted] (P14) -- +(-.3,.7);
 \draw [dotted] (P15) -- +(.3,.7);
\end{tikzpicture} \]
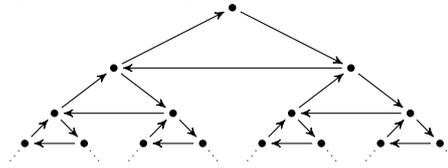
\captionof{figure}{The quiver \(\mathcal Q\)} \label{figure.an}
\end{minipage}

For type $A_n$, the cluster tilted algebras can be described as follows:
\begin{itemize}
\item all non-trivial cycles are oriented and have length 3.
\item any vertex has at most four neighbours
\item if a vertex has four neighbours then two of the arrows adjacent to the vertex are part of a $3$-cycle, and the other two arrows also adjacent to the vertex are part of another $3$-cycle.
\item if a vertex has three neighbours then two of the adjacent arrows belong to a $3$-cycle and the last arrow adjacent to the vertex does not belong to any $3$-cycle.
\end{itemize}

The next result is not explicitly stated in \cite{bow}, but follows directly from results in \cite{bow}. It shows that the position of a vertex in the quiver of a cluster tilted algebra is closely related to which \(\tau\)-orbit the corresponding indecomposable summand lies in, in the AR-quiver of the cluster catgory. This will be very useful in section \ref{sectA}.  

\begin{theorem}\label{bowPerp}
Let \(Q\) be a quiver of a \(2\)-cluster tilting object \(T\) of \(\mathcal C_2(A_n)\). Let \(i\) be a vertex of \(Q\), such that deleting \(i\) from \(Q\) gives rise to two connected subquivers \(Q_1\) and \(Q_2\), with \(|Q_1|\leq|Q_2|\). Then the indecomposable summand \(T_i\) of \(T\) corresponding to vertex \(i\) lie in \(\tau\)-orbit nr \(|Q_1|+1\), counted from the outermost \(\tau\)-orbit of \(\mathcal C_2(A_n)\).
\end{theorem}

\begin{proof}
Suppose that \(T_i\) lies in the \(j\) outermost \(\tau\)-orbit. Then by proposition 2.2 \cite{bow} this means that deleting the vertex \(i\) from \(Q\) one gives rise to connected subquivers \(Q_1^j\) and \(Q_2^j\) of size \(n-j\) and \(j-1\). Hence we must have \(j=i\).  
\end{proof}

\begin{theorem}(lemma 3.2 in \cite{bow})
Any \(2\)-cluster tilting object \(T\) in \(\mathcal C_2(A_m)\) is induced by a tilting module over \(kA_m\) where \(A_m\) has linear orientation. 
\end{theorem}

The next result is important as it gives us an easy method of determining the summands of all cluster tilting objects of \(\mathcal C_2(A_n)\) having a certain quiver \(Q\) as the quiver of the cluster tilted algebra, given that we know the summands for one such cluster tilting object.

\begin{theorem}\cite{HermundA}\label{HermundA}
Let $T$ and $T'$ be two cluster tilting objects of type $A_n$. The cluster tilted algebras $\End_{\mathcal{C}}(T)^{op}$ and $\End_{\mathcal{C}}(T')^{op}$ are isomorphic if and only if $T$=$\tau^i T'$ for some integer $i$.
\end{theorem}

\subsection{Cluster tilted algebras of type D}\label{ClTiltAlgTypeD}
We will base the summary of type \(D\) on the description given in \cite{bow}. Given a quiver $D_n$ then the quivers of the  cluster tilted algebras of $D_n$ are of the following forms:
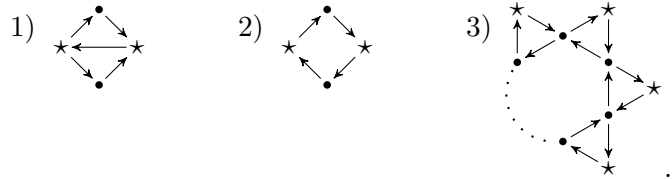
\[ \begin{tikzpicture}[scale=.5,yscale=-1]
 \pgftransformshift{\pgfpoint{-2cm}{0}};
 \node at (1,.5) {1)};
 \node (P1) at (2,1) [inner sep=1pt] {$\star$};
 \node (P2) at (3,0) [vertex] {};
 \node (P3) at (3,2) [vertex] {};
 \node (P4) at (4,1) [inner sep=1pt] {$\star$};
 \draw [->] (P1) -- (P2);
 \draw [->] (P1) -- (P3);
 \draw [->] (P2) -- (P4);
 \draw [->] (P3) -- (P4);
 \draw [->] (P4) -- (P1);
 
 \pgftransformshift{\pgfpoint{6cm}{-5cm}};
 \node at (1,5.5) {2)};
 \node (P1) at (2,6) [inner sep=1pt] {$\star$};
 \node (P2) at (3,5) [vertex] {};
 \node (P3) at (3,7) [vertex] {};
 \node (P4) at (4,6) [inner sep=1pt] {$\star$};
 \draw [->] (P1) -- (P2);
 \draw [->] (P3) -- (P1);
 \draw [->] (P2) -- (P4);
 \draw [->] (P4) -- (P3);
 \pgftransformshift{\pgfpoint{6cm}{-3cm}};
 \node at (1,8.5) {3)};
 \node (P1) at (2,9.4) [vertex] {};
 \node (P2) at (3.2,8.7) [vertex] {};
 \node (P3) at (4.4,9.4) [vertex] {};
 \node (P4) at (4.4,10.8) [vertex] {};
 \node (P5) at (3.2,11.5) [vertex] {};
 \node (P6) at (2,8) [inner sep=1pt] {$\star$};
 \node (P7) at (4.4,8) [inner sep=1pt] {$\star$};
 \node (P8) at (5.6,10.1) [inner sep=1pt] {$\star$};
 \node (P9) at (4.4,12.2) [inner sep=1pt] {$\star$};
 \draw [->] (P2) -- (P1);
 \draw [->] (P3) -- (P2);
 \draw [->] (P4) -- (P3);
 \draw [->] (P5) -- (P4);
 \draw [->] (P1) -- (P6);
 \draw [->] (P6) -- (P2);
 \draw [->] (P2) -- (P7);
 \draw [->] (P7) -- (P3);
 \draw [->] (P3) -- (P8);
 \draw [->] (P8) -- (P4);
 \draw [->] (P4) -- (P9);
 \draw [->] (P9) -- (P5);
 \draw [thick, loosely dotted] (P1) .. controls (1.4,10.45) and (2,11.5) .. (P5);
\end{tikzpicture} .\]
Each star represents what is called a connecting vertex. At each connecting vertex there is attached a quiver from the mutation class of Dynkin type \(A_l\), note that the case \(l=0\) can occur.

We now go on to distinguish between two different types of indecomposable objects in \(\mathcal C_2(D_n)\). This will enable a more accurate description of the distribution of the indecomposable summands of each of the the three types of quivers of cluster tilted algebras. One should take note of the Ar-quiver of \(\mathcal D^b(D_n)\), as shown in figure \ref{AR-quiverD}.

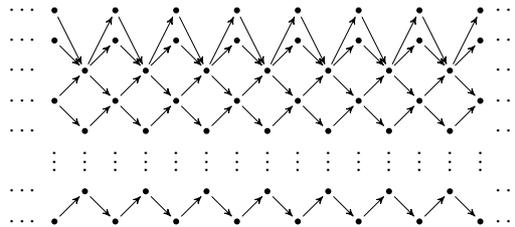
\begin{figure}[h]
\[ \scalebox{.8}{ \begin{tikzpicture}[scale=.5,yscale=-1]
 \foreach \x in {0,...,7}
  \foreach \y in {1,2,4,8}
   \node (\y-\x) at (\x*2,\y) [vertex] {};
 \foreach \x in {0,...,6}
  \foreach \y in {3,5,7}
   \node (\y-\x) at (\x*2+1,\y) [vertex] {};
 \foreach \xa/\xb in {0/1,1/2,2/3,3/4,4/5,5/6,6/7}
  \foreach \ya/\yb in {1/3,2/3,4/3,4/5,8/7}
   {
    \draw [->] (\ya-\xa) -- (\yb-\xa);
    \draw [->] (\yb-\xa) -- (\ya-\xb);
   }
 \foreach \x in {0,...,14}
  \node at (\x,5.8) {$\vdots$};
 \foreach \y in {1,...,5,7,8}
  {
   \node at (-1,\y) {$\cdots$};
   \node at (15,\y) {$\cdots$};
  }
\end{tikzpicture} } \]\caption{The AR-quiver of \(\mathcal D^b(D_n)\).}\label{AR-quiverD}
\end{figure}

\begin{definition}
The objects lying in the two top \(\tau\)-orbits of the AR-quiver of \(\mathcal D^b(D_n)\) as illustrated above, are called \(\alpha\)-objects. Objects in other \(\tau\)-orbits are called \(\beta\)-objects.  Also we define the flip \(\phi T_i\) of an \(\alpha\)-object \(T_i\) to be the other \(\alpha\)-object that is a summand of the middle-term in the same AR-triangle as \(T_i\). 
\end{definition} 

The images of \(\alpha\) objects, respectively \(\beta\) objects, in \(\mathcal C_2(D_n)\) are also called \(\alpha\) objects, respectively \(\beta\) objects. 

From \cite{bow} we have the following theorem, giving us information about the distribution of \(\alpha\)- objects for each type of quiver of type \(D\).

\begin{theorem}(theorem 4.1 \cite{bow})\label{alphaDistribution}
Let \(U\) be a \(2\)-cluster tilting object of type \(D_n\). Let \(Q\) be the quiver of the corresponding cluster tilted algebra. Then
\begin{itemize}
\item If \(Q\) is of type 1 listed above, then \(U\) has exactly two \(\alpha\)-objects \(U_{\alpha}\) and \(\phi _{\alpha}(U_{\alpha})\) as summands. 
\item If \(Q\) is of type 2 listed above, then \(U\) has exactly two \(\alpha\)-objects \(U_1\) and \(U_2\) as summands, and \(U_2\neq\phi U_1\). Furthermore if the size of the subquivers at the connecting vertices are \(n_1\) and \(n_2\) then \(U_1\) and \(U_2\) are such that \(U_1=\tau^{n_1+1}U_2\) or \(U_1=\phi\tau^{n_1+1}U_2\) and \(U_2=\tau^{n_2+1}U_1\) or \(U_2=\phi\tau^{n_2+1}U_1\)
\item If \(Q\) is of type 3 then then \(U\) has more than two \(\alpha\)-objects as summands. The \(\alpha\)-objects are distributed within the AR-quiver depending on the size of the subquivers of type \(A\) at the various connecting vertices, as described for type 2. 
\end{itemize}
\end{theorem}

Consider a \(2\)-cluster tilting object \(T\) of \(\mathcal C_2(D_n)\), with quiver \(Q\). Then it follows from the proof in \cite{bow} that if \(Q_l\) is a quiver in the mutation class of \(A_l\) that is attached at a connecting vertex of \(Q\), then the indecomposable summands of \(T\) corresponding to the vertices in \(Q_l\) lie in a subcategory of \(\mathcal C_2(D_n)\) that is equivalent to the category \(\modf kA_l\) where \(A_l\) has linear orientation. Subcategories of \(\mathcal C_2(D_n)\) that are equivalent to \(\modf kA_n\) for some \(n\) with linear orientation on \(A_n\) will be called an \(A\)-triangle of size \(n\). We will encounter such subcategories again in sections \ref{sectA} and \ref{typeD}. Furthermore, if \(\Delta\) is an \(A\)-triangle of size \(l\) in \(\mathcal C_2(D_n)\), then we will denote the indecomposable object in \(\mathcal C_2(D_n)\) corresponding to the projective-injective object of \(\modf kA_l\) by \(\Pi(\Delta)\).

\begin{lemma}\label{MorpAtriangle}\cite{bow}
Let \(T\) be \(2\)-cluster tilting object in \(\mathcal C_2(D_n)\), and let \(\Delta\) be an \(A\)-triangle in \(\mathcal C_2(D_n)\). Then all morphism between an indecomposable object in \(\Delta\) and an indecomposable summand of \(T\) not in \(\Delta\) must factor through the projective-injective object \(\Pi(\Delta)\) of \(\Delta\). 
\end{lemma}

The last result of this section is the equivalent of theorem \ref{HermundA} for type D. 
\begin{theorem}\cite{HermundD}\label{hermundD}
Let $T$ and $T'$ be two cluster tilting objects of type $D_n$. Then the cluster tilted algebras $\End_{\mathcal{C}}(T)^{op}$ and $\End_{\mathcal{C}}(T')^{op}$ are isomorphic if and only if $T$=$\phi^j\tau^iT'$ for some integers $i$ and $j$.
\end{theorem}

\section{$2$-cluster tilting subcategories of type $A$}\label{sectA}

In this section we will apply the method described in the last part of section \ref{theory} to the cases where \(H\) is a path-algebra of type \(kA_n\), where \(A_n\) is some orientation of the Dynkin diagram \(A_n\). As discussed in section \ref{theory}, it is sufficient to study under which functors of the form \(\tau^s\) the cluster tilting objects of the cluster category \(\mathcal C_2(A_n)\) are periodic, where \(s\in\mathbb Z\). We give a complete overview of the smallest possible positive value of \(s\) for each cluster tilting object of \(\mathcal C_2(A_n)\). Furthermore, we show that the periodicity of a \(2\)-cluster tilting object depends on symmetric properties of the quiver of the corresponding cluster tilted algebra. 

First we need to recall some facts about the AR-quiver of the \(2\)-cluster category of \(\mathcal C_2(A_n)\). Recall that if \(n\) is even then all the \(\tau\)-orbits in  \(\mathcal C_2(A_n)\) have \(n+3\) objects. If \(n=2l+1\) is odd, then the innermost \(\tau\)-orbit has \(l+2\) objects, and all the other \(\tau\)-orbits have \(n+3\)-objects. Hence for any indecomposable object \(Z\) in \(\mathcal C_2(A_n)\) we have \(\tau^{(n+3)}Z=Z\) (for \(n\) even and odd). We first give a result narrowing down the possible values of \(s\) for which it is interesting to study periodicity of \(\tau^s\) on the \(2\)-cluster tilting subcategories.

\begin{theorem}\label{valuesOfs}
Let \(T\) be a \(2\)-cluster tilting object in \(\mathcal C_2(A_n)\) such that \(\tau^sT=T\) for some \(0<s<n+3\). Then
\begin{itemize}
\item[1.] either \(s=\frac{n+3}{2}\), with \(n\) an odd number and \(T_i\) has exactly one summand in the innermost \(\tau\)-orbit
\item[2.] or \(s=\frac{n+3}{3}\) and \(n\) is divisible by \(3\).
\end{itemize}
\end{theorem}

\begin{proof}
We will first assume that all indecomposable summands of \(T\) lie in \(\tau\)-orbits in the AR-quiver of \(\mathcal C_2(A_n)\)  with \(n+3\) indecomposable objects. Since we have \(\tau^{n+3}T=T\) and have assumed \(\tau^s T= T\), we must have that \(s\) is a factor of \(n+3\), i.e. \(n+3=sa\) . Then \(a=\frac{n+3}{s}\) must be a factor of \(n\), since each \(\tau\)-orbit containing one indecomposable summand of \( T\) must contain a multiple of \(a\) indecomposable summands of \(T\), and \(T\) has \(n\) indecomposable summands. That is \((n+3)=sa\) and \(n=ab\) for some \(b\in\mathbb N\). Hence \(a\in\left\{1,3\right\}\). However by assumption \(s<n+3\), so that \(a>1\), therefore we have \(a=3\) and the second part follows. 

For the first part, assume that there is at least one indecomposable summand \(T_i\) of \(T\) that lie in a \(\tau\)-orbit containing \(\frac{n+3}{2}\) indecomposable objects (clearly this can only happen when \(n\) is odd). Then we have \(\tau^{\frac{n+3}{2}}T_i=T_i\). Note that this is the minimal value of \(s\) such that \(\tau^sT_i=T_i\). This is most easily seen in a figure. Drawing the \(\Ext\)-support of \(T_i\) will reveal that it has extension to all other objects in the innermost \(\tau\)-orbit. Hence the claim is proved.
\end{proof}

We will study the first case of theorem \ref{valuesOfs} closer in subsection \ref{2sym}, and the second case in subsection \ref{3sym}.

\subsection{2-symmetric $2$-cluster tilting objects of type A}\label{2sym}
In this subsection we will study \(2\)-cluster tilting objects of \(\mathcal C_2(A_n)\) which are covered by part \(1\) of theorem \ref{valuesOfs}. It is clear from theorem \ref{valuesOfs} that if \(U\) is a \(2\)-cluster tilting object of \(\mathcal C_2(A_n)\) such that \(\tau^{\frac{n+3}{2}} U= U\) in \(\mathcal C_2(A_n)\), then \(U\) has one indecomposable summand lying in the innermost \(\tau\)-orbit. Our aim is to describe the quiver \(Q\) of \(\End_{\mathcal C}(U)^{op}\) when \(U\) is such that \(\tau^{\frac{n+3}{2}}U=U\). We start with an immediate consequence of \(U\) having an indecomposable summand in the innermost \(\tau\)-orbit.

\begin{theorem}\label{decompU}
Let \( U\) be a basic \(2\)-cluster tilting object of  \(\mathcal C_2(A_n)\) with a summand \(U_{l+1}\) in the innermost \(\tau\)-orbit.  Then the quiver of \(\End_{\mathcal C}(U)^{op}\) is of the form illustrated in figure \ref{QuiverAwmiddle} , where \(Q_1\) and \(Q_2\) are quivers of cluster tilted algebras of type \(A_{\frac{n+1}{2}}\), and \(Q_1\) and \(Q_2\) have exactly one vertex in common, namely the vertex corresponding to the summand \(U_{l+1}\).
\end{theorem}

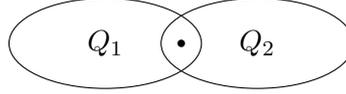
\begin{figure}[h]
\begin{center}
\begin{tikzpicture}
\node(M) at (0,0)[vertex]{}; 
\draw (1,0) ellipse (36pt and 17pt);
\draw (-1,0) ellipse (36pt and 17pt);
\node(1) at (-1,0){$Q_1$};
\node(2) at (1,0){$Q_2$};
\end{tikzpicture}
\end{center}\caption{Quiver of type \(A_n\) with one summand in innermost \(\tau\)-orbit}\label{QuiverAwmiddle}
\end{figure}

\begin{proof}
Direct from \cite{bow} we have that \( U\) may be considered as \( U=U_1\oplus  U_{l+1}\oplus U_2\) where \( U_{l+1}\) lies in the innermost \(\tau\)-orbit of \(\mathcal C_2(A_n)\), \( U_1\) and \( U_2\) are \(2\)-cluster tilting objects of type \(A_{\frac{n-1}{2}}\). Since \(Q\) is connected \( U_1\oplus U_{l+1}\) and \( U_{l+1}\oplus U_2\) is connected of type \(A_{\frac{n+1}{2}}\). The distribution of summands is illustrated in figure \ref{figMiddle}.
\end{proof}

\begin{figure}
\[ \scalebox{.99} { \begin{tikzpicture}[scale=.5,yscale=-1]
 \fill [fill1] (0,.5) -- (.5,.5) -- (5,5) -- (9.5,.5) -- (10,.5) -- (10,7.5) -- (9.5,7.5) -- (5,3) -- (.5,7.5) -- (0,7.5) -- cycle;
 \foreach \x in {0,...,5}
  \foreach \y in {1,3,5,7}
   \node (\y-\x) at (\x*2,\y) [vertex] {};
 \foreach \x in {0,...,4}
  \foreach \y in {2,4,6}
   \node (\y-\x) at (\x*2+1,\y) [vertex] {};
 \replacevertex[fill1]{(4-2)}{[tvertex] {$ U_{l+1}$}}
 \foreach \xa/\xb in {0/1,1/2,2/3,3/4,4/5}
  \foreach \ya/\yb in {1/2,3/4,5/4,7/6}
   {
    \draw [->] (\ya-\xa) -- (\yb-\xa);
    \draw [->] (\yb-\xa) -- (\ya-\xb);
   }
 \foreach \xa/\xb in {0/1,1/2,2/3,3/4,4/5}
  \foreach \ya/\yb in {3/2,5/6}
   {
    \draw [thick,loosely dotted] (\ya-\xa) -- (\yb-\xa);
    \draw [thick,loosely dotted] (\yb-\xa) -- (\ya-\xb);
   }
 \draw [dashed] (0,.5) -- (0,7.5); 
 \draw [dashed] (10,.5) -- (10,7.5);
 \draw [decorate,decoration=brace] (1.5,.5) -- node [above] {$C_{\frac{n-1}{2}}$} (8.5,.5);
 \draw [decorate,decoration={brace,mirror}] (1.5,7.5) -- node [below] {$C_{\frac{n-1}{2}}$} (8.5,7.5);
\end{tikzpicture} } \]
\caption{figure} \label{figMiddle}
\end{figure}

\begin{definition}
Let \(U\) be a \(2\)-cluster tilting object of \(\mathcal C_2(A_n)\) with quiver \(Q\), where \(n=2l+1\) and one summand \(U_{l+1}\) of \(U\) lie in the innermost \(\tau\)-orbit of \(\mathcal C_2(A_n)\). Then we call \( U\) a \(2\)-symmetric \(2\)-cluster tilting object if \(Q_1=Q_2\), (i.e. if the quiver \(Q\) is mirror-symmetric about the vertex corresponding to \(U_{l+1}\)).
\end{definition}

As we will go on to show, a \(2\)-cluster tilting object \(T\) of \(\mathcal C_2(A_n)\) is \(2\)-symmetric if and only if \(\tau^{(n+3)/2})T=T\). In order to show this we will first give a characterization of pairs of indecomposable summands that are closed under \(\tau^{(n+3)/2}\). 

\begin{definition}\label{vertically}
Let \(M\) and \(N\) be two indecomposable objects of \(\mathcal C_2(A_n)\). We call $M$ and $N$ vertically aligned if there exists a  sequence of indecomposable objects $M$, $M_1$, $M_2$, $\ldots$ ,$M_t$, $N$ where $M\oplus M_1$ is the middle term of an AR-sequence, $M_i\oplus M_{i+1}$ is the middle term of an AR-sequence for each $i\in \left\{1, \ldots t\right\}$, and lastly $M_t\oplus N$ is the middle term of an AR-sequence.
\end{definition}

Note that two vertically aligned indecomposable objects lying in the same \(\tau\)-orbit of the AR-quiver, have to live in an \(\tau\)-orbit with \(n+3\) indecomposable objects.

\begin{lemma}\label{pairwiseVertically}
Let $n=2l+1$, and let $T_i$ and $T_j$ be indecomposable vertically aligned objects lying in the $\tau$-orbit of $\mathcal{C}^2(A_n)$. Then $\tau^{l+2}T_i=T_j$ and $\tau^{l+2}T_j=T_i$.
\end{lemma}

\begin{proof}
Since \(T_i\) and \(T_j\) are vertically aligned, they can not be in the innermost \(\tau\)-orbit. Without loss of generality we may assume that \(T_i\) and \(T_j\) lie in the part of the category corresponding to the module category of \(kA_n\) with linear orientation. This will make it possible to count how many objects lie between \(T_i\) and \(T_j\), by first counting within the AR-quiver of \(kA_n\) and then adding the extra objects corresponding to the projectives shifted once. Let the row of the AR-quiver of \(kA_n\) with one object be row nr \(1\), and the row with \(n\) objects be row nr \(n\), i.e. row nr \(1\) and row nr \(n\) are both contained in the outermost \(\tau\)-orbit of the AR-quiver of \(\mathcal C_2(A_n)\).

Assume now that \(T_i\) is located in row nr \(d\) of the AR-quiver of \(kA_n\), where \(0<d<l+1\). Moreover, assume that within this row of the AR-quiver of \(kA_n\) there are \(s\) objects to the right of \(T_i\) and \(d-s-1\) objects to the left of \(T_i\) (total of \(d\) objects).

Since \(T_i\) and \(T_j\) are assumed to be vertically aligned and in the same \(\tau\)-orbit of the AR-quiver of \(\mathcal C_2(A_n)\), the location of \(T_j\) correspond to being in row \(n-d+1\) in the AR-quiver of \(kA_n\). Furthermore, row \(n-d+1\) has \(n-d+1-d\) more objects than row nr \(d\). Hence there is \(\frac{1}{2}(n-2d+1)+(d-s-1)\) objects to the left of \(T_j\) within the row \(n-d+1\) of the AR-quiver of \(kA_n\), and there are \(s+\frac{1}{2}(n-2d+1)\) objects to the right of \(T_j\).

Starting from \(T_i\) and counting indecomposable objects along the \(\tau\)-orbit towards \(T_j\), there are 
\(s+1+\frac{1}{2}(n-2d+1)+(d-s-1)=l+1\) objects between \(T_i\) and \(T_j\). Hence \(\tau^{l+2}T_i=T_j\).
Counting along the \(\tau\)-orbit from \(T_j\) towards \(T_i\) yields the same equations and hence \(\tau^{l+2}T_j=T_i\). 
\end{proof}

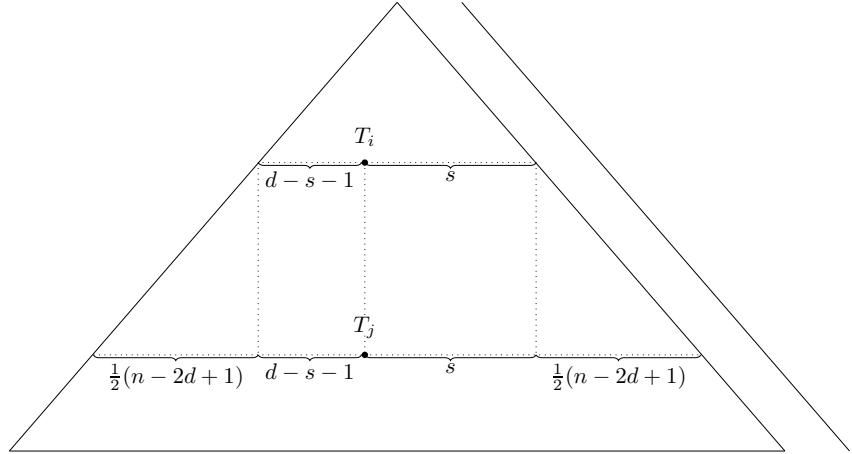
\begin{figure}[h]
\begin{center}
\scalebox{.85}{
\begin{tikzpicture}
\node(O) at (-0.5,1.5)[vertex]{};
\node[above] at (O.north)[]{\small $T_i$};
\node(N) at (-0.5,-1.5)[vertex]{};
\node[above] at (N.north)[]{\small $T_j$};
\draw(-6,-3)--(0,4); 
\draw(6,-3)--(0,4);
\draw(1,4)--(7,-3);
\draw(-6,-3)--(6,-3);
\draw[dotted](-0.55,1.5)--(-2.15,1.5);
\draw[decorate,decoration=brace](-0.55,1.5)--node[anchor=north]{\small $d-s-1$}(-2.15,1.5);
\draw[decorate,decoration=brace] (-0.55,-1.5)--node[anchor=north]{\small $d-s-1$}(-2.15,- 1.5);
\draw[dotted](-0.55,-1.5)--(-4.7,-1.5);
\draw[decorate,decoration=brace](-2.15,-1.5)--node[anchor=north]{\small $\frac{1}{2}(n-2d+1) $}(-4.7,-1.5);
\draw[dotted](-2.15,1.5)--(-2.15,-1.5);
\draw[dotted](-0.5,1.5)--(-0.5,-0.9);
\draw[dotted](-0.5,-1.2)--(-0.5,-1.5);
\draw[decorate,decoration=brace](2.15,1.5)--node[anchor=north]{\small $s$}(-0.5,1.5);
\draw[dotted](2.15,1.5)--(-0.5,1.5);
\draw[dotted](2.15,1.5)--(2.15,-1.5);
\draw[decorate,decoration=brace](2.15,-1.5)--node[anchor=north]{\small $s$}(-0.5,-1.5);
\draw[dotted](-0.5,-1.5)--(4.7,-1.5);
\draw[decorate,decoration=brace](4.7,-1.5)--node[anchor=north]{\small $\frac{1}{2}(n-2d+1)$}(2.15,-1.5);
\end{tikzpicture}}
\end{center}\caption{Schematic figure of \(\mathcal C_2(A_n)\). Objects \(T_i\) and \(T_j\) lie in the \(d\)-outermost \(\tau\)-orbit.}\label{skjmatiskA}
\end{figure}

\begin{theorem}\label{thm18}
Let \(n=2l+1\) and let \(T_i\) and \(T_j\) be indecomposable objects lying in a \(tau\)-orbit of \(\mathcal C_2(A_n)\) containing \(n+3\) objects. Then \(T_i\) and \(T_j\) are vertically aligned if and only if \(\tau^{l+2}T_i=T_j\) and \(\tau^{l+2}T_j=T_i\).
\end{theorem}

\begin{proof}
One implication is shown in lemma \ref{pairwiseVertically}. Assume that \(T_i\) and \(T_j\) are indecomposable objects such that \(\tau^{\frac{n+3}{2}}T_i=T_j\) and \(\tau^{\frac{n+3}{2}}T_j=T_i\). Since \(n\) is odd \(T_j\) and \(T_i\) must be vertically aligned. 
\end{proof}

\begin{theorem}
Let \( U\) be a basic \(2\)-cluster tilting object of \(\mathcal C_2(A_n)\) where \(n\) is odd. Then \(\tau^{\frac{n+3}{2}} U= U\) if and only if \(U\) is \(2\)-symmetric. 
\end{theorem}

\begin{proof}
First assume that \( U\) is \(2\)-symmetric. We now want to determine the distribution of the indecomposable summands of \( U\) within the AR-quiver of \(\mathcal C_2(A_n)\). There is one summand \(U_{l+1}\) in the innermost \(\tau\)-orbit, for which it is clear that \(\tau^{\frac{n+3}{2}}U_{l+1}=U_{l+1}\).

By  the proof of theorem \ref{decompU} we can write \( U\) on the form \( U= U_1\oplus U_{l+1}\oplus U_2\) where \(U_{i+1}\) lies in the innermost \(\tau\)-orbit. Choose a basic tilting object \(T\) of \(\modf kA_{l+1}\) with the same quiver as \(\End_{\mathcal C}( U_1\oplus U_{l+1})\). Let \(U_{l+1}\) be identified with the projective-injective object of \(\modf kA_{l+1}\), and the other indecomposables of \(T\) with the corresponding vertices in the subcategory of type \(A_{l+1}\) determined by \(U_{l+1}\). By choosing the same tilting object \(T\) again we determine indecomposables in the AR-quiver that correspond to \(U_2\). The indecomposables of \(U_1\) and \(U_2\) are pairwise vertically aligned. By theorem \ref{thm18} we are done. 

Now assume that \(\tau^{\frac{n+3}{2}}U=U\). By theorem \ref{thm18} the indecomposable summands of \(U\) are pairwise vertically aligned. Since \(U\) has an odd number of indecomposable summands, there is one indecomposable summand \(U_{l+1}\) of \(U\) lying in the innermost \(\tau\)-orbit. Hence the quiver \(Q\) of \(\End_{\mathcal C}(U)^{op}\) is of the shape described by figure \ref{QuiverAwmiddle}, and \(U\) is a \(2\)-symmetric \(2\)-cluster tilting object in \(\mathcal C_2(A_n)\). 
\end{proof}

\begin{corollary}
Let $T$ be a \(2\)-symmetric $2$-cluster tilting object of type $A_n$ and $\mathcal{X}$ be the preimage in the bounded derived category. Then  \(\tau^{\frac{n+3}{2}}\mathcal X=\mathcal X\).
\end{corollary}

\begin{proof}
Follows from lemma \ref{periodInDerived}.
\end{proof}

\subsection{3-symmetric \(2\)-cluster tilting objects of type A}\label{3sym}

We will now focus on the second case of theorem \ref{valuesOfs}. We give a description of the quivers of the corresponding to the \(2\)-cluster tilting objects \(\mathcal T\) such that \(n=3l\) and \(\tau^{l+1}\mathcal T=\mathcal T\). 

\begin{definition}
Let $T$ be a $2$-cluster tilted object of type $\mathbb{A}_n$ where $n=3l$ and $Q$ is the quiver of the corresponding cluster tilted algebra $\End(T)^{op}$. Then we call $T$ $3$-symmetric if the quiver of $\End(T)^{op}$ has a central 3-cycle with the same subquiver $Q_l$ of type $A_{l}$ attached at each vertex of the 3-cycle (Each subquiver of size $A_l$ contains one of the vertices in the central $3$-cycle). 
\end{definition}

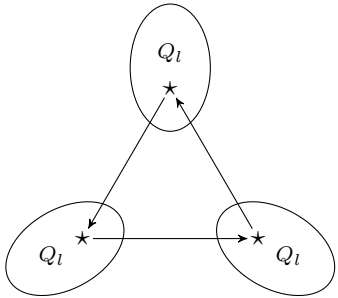
\begin{figure}[b]
\begin{tikzpicture}
\node(P1) at (0,1.3335) [inner sep=1pt] {$\star$}; 
\node(P2) at (-1.155,-0.6665)[inner sep=1 pt] {$\star$};
\node(P3) at (1.155,-0.6665)[inner sep=1 pt]{$\star$};
\draw [->] (P1) -- (P2);
\draw [->](P2)--(P3);
\draw [->](P3)--(P1);
\draw [rotate=30](0,0) (-1.6,0) ellipse (24pt and 15 pt);
\node (S1) at (30: -1.8cm)[tvertex]{$Q_l$};
\draw [rotate=150](0,0) (-1.6,0) ellipse (24pt and 15 pt);
\node (S_2) at (150:-1.8 cm)[tvertex]{$Q_l$};
\draw [rotate=270](0,0) (-1.6,0) ellipse (24pt and 15 pt);
\node(S_3) at (90: 1.8 cm)[tvertex]{$Q_l$};
\end{tikzpicture}\caption{The general shape of the quiver of a \(3\)-symmetric \(2\)-cluster tilting object of type \(A_{3l}\).}\label{3symDef}
\end{figure}

\begin{theorem}\label{3sumImp1}
Let \(U\) be a \(3\)-symmetric \(2\)-cluster tilting object of \(\mathcal C_2(A_n)\) where \(n=3l\). Then \(\tau^{l+1}U=U\).
\end{theorem}

\begin{proof}
Let \(U\) be a \(3\)-symmetric \(2\)-cluster tilting object of \(\mathcal C_2(A_n)\) where \(n=3l\). Denote by \(Q\) the quiver of \(\End_{\mathcal C}(U)^{op}\) and by \(Q_l\) the subquivers of \(Q\) as indicated in figure \ref{3symDef}.
We will first define a set of indecomposable objects  \(V\) of \(\mathcal C_2(A_n)\). Then we will show that \(V\) is a \(2\)-cluster tilting object with quiver \(Q\) (the same quiver as \(U\)), and \(\tau^{l+1}V=V\). By Theorem \ref{HermundA} this completes the proof.

We now describe the set \(V\) of \(n\) indecomposable objects in \(\mathcal C_2(A_n)\). Start by choosing a random object in the \(l\) outermost \(\tau\)-orbit, and name this object \(V_l\). Then choose \(V_{2l}:=\tau^{l+1}V_l\) and \(V_{3l}:=\tau^{l+1}V_{2l}=\tau^{2l+2}V_l\). The objects \(V_l,V_{2l}\) and \(V_{3l}\) determine 3 Abelian subcategories  of type \(A_l\)(where \(A_l\) has linear orientation) in \(\mathcal C_2(A_{3l})\). We denote the respective subcategories by \(\Delta_1,\Delta_2\) and \(\Delta_3\). An illustration of this situation is given in figure \ref{figure.oppe_og_nede}.

The projective-injective object \(\Pi(\Delta)\) of \(\Delta_i\) correspond to the object \(V_{il}\). In \(\Delta_1\) we choose a tilting object \(V_{1\ldots l}\) with quiver \(Q_l\). Note that \(V_l\) will be a summand of \(V_{1\ldots l}\) as it corresponds to the projective injective object. Then we choose the corresponding tilting objects \(V_{l+1\ldots 2l}\) in \(\Delta_2\) and \(V_{2l+1\ldots 3l}\) in \(\Delta_3\), containing respectively \(V_{2l}\) and \(V_{3l}\). We then define \(V:=V_{1\ldots l}\oplus V_{l+1\ldots 2l}\oplus V_{2l+1\ldots 3l}\).

From figure \ref{figure.oppe_og_nede} it is clear that the objects \(V_l, V_{2l}\) and \(V_{3l}\) are chosen in such a way that they are compatible in the same \(2\)-cluster tilting object. There are no extensions between any objects in \(\Delta_1\) and \(\Delta_2\), \(\Delta_2\) and \(\Delta_3\) and \(\Delta_3\) and \(\Delta_1\) due to the choice of \(V_l,V_{2l}\) and \(V_{3l}\). Furthermore, there are no extensions between any indecomposable summands within \(\Delta_1, \Delta_2\) or \(\Delta_3\) since \(V_{1\ldots l}, V_{l+1\ldots 2l}\) and \(V_{2l+1\ldots 3l}\) are tilting objects in respective subcategories. Since \(V\) has exactly \(n\) indecomposable summands, we conclude that it is indeed a \(2\)-cluster tilting object of \(\mathcal C_2(A_n)\). 

From the method used to choose \(V_{1\ldots l}, V_{l+1\ldots 2l}\) and \(V_{2l+1\ldots 3l}\) we have that \(\tau^{l+1} V= V\). Hence the theorem is proved.  
\end{proof}


\begin{minipage}{\textwidth}
\[ \scalebox{1} { \begin{tikzpicture}[scale=.5,yscale=-1]
 \fill [fill1] (0,2)-- (1.5,.5) --(2.5,.5) -- (7,5) -- (11.5,.5) -- (12.5,.5)-- (15,3)--(15,11)-- (7,3)--(0,10)--cycle; 
\fill [fill2] (0,10)--(2,8)-- (6.5,12.5)-- (7.5,12.5)--(15,5)--(15,3)--(12.5,.5)--(11.5,.5)--(2,10)--(0,8)--cycle;
\fill [fill12]  (15,7)--(12,10)--(2.5,.5)--(1.5,.5)--(0,2)--(0,6)--(6.5,12.5)--(7.5,12.5)--(12,8)--(15,11)--cycle;
 \foreach \x in {0,...,7}
  \foreach \y in {1,3,5,7,9,11}
   \node (\y-\x) at (\x*2,\y) [vertex] {};
 \foreach \x in {0,...,7}
  \foreach \y in {2,4,6,8,10,12}
   \node (\y-\x) at (\x*2+1,\y) [vertex] {};
 \replacevertex[fill2]{(9-1)}{[tvertex] {$V_{2l}$}}
 \replacevertex[fill12]{(9-6)}{[tvertex]{$V_{1l}$}}
 \replacevertex[fill1]{(4-3)}{[tvertex]{$V_{3l}$}}
 
 \foreach \xa/\xb in {0/1,1/2,2/3,3/4,4/5,5/6,6/7}
  \foreach \ya/\yb in {1/2,3/4,5/4,7/6,9/8,9/10,11/12}
   {
    \draw [->] (\ya-\xa) -- (\yb-\xa);
    \draw [->] (\yb-\xa) -- (\ya-\xb);
   } 
 \foreach \ya/\yb in {1/2,3/4,5/4,7/6,9/8,9/10,11/12}
 	{ \draw[->](\ya-7)--(\yb-7);}
 \foreach \ya/\yb in {3/2,5/6,7/8,11/10}
 	{\draw[thick,loosely dotted] (\ya-7)--(\yb-7);} 	
 \foreach \xa/\xb in {0/1,1/2,2/3,3/4,4/5,5/6,6/7}
  \foreach \ya/\yb in {3/2,5/6,7/8,11/10}
   {
    \draw [thick,loosely dotted] (\ya-\xa) -- (\yb-\xa);
    \draw [thick,loosely dotted] (\yb-\xa) -- (\ya-\xb);
   }
 \draw [dashed] (0,.5) -- (0,12.5); 
 \draw [dashed] (15,.5) -- (15,12.5);
 \draw [decorate, decoration={brace}](-0.25,12.5) -- node [left]{$l$}(-0.25,8.5);
 \draw [decorate, decoration={brace}](-0.25,4.5) -- node [left]{$l$}(-0.25,0.5);
 \draw[rounded corners=8pt] (6,0.5)--node[anchor=south]{$\Delta_3$}(8,0.5)--(11.25,0.5)--(7,4.75)--(2.75,0.5)--(6,0.5); 
\end{tikzpicture} } \]

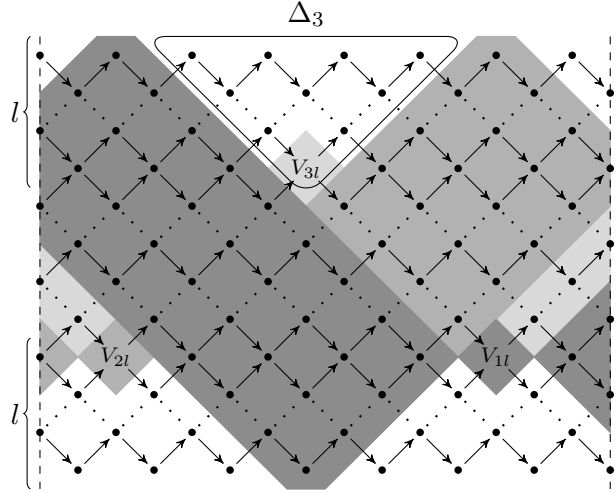
\captionof{figure}{The indecomposables on a white background may be part of the same \(2\)-cluster tilting object as \(V_{l}, V_{2l}\) and \(V_{3l}\). The indecomposables with the same background color as \(V_{3l}\) or darker may not be part of a \(2\)-cluster tilting object together with \(V_{3l}\), and the same is also true for \(V_{2l}\) and \(V_l\).} \label{figure.oppe_og_nede}
\end{minipage}

\begin{lemma}\label{middleEmpty}
Let \(n=3l\) and \(U\) be a \(2\)-cluster tilting object of \(\mathcal C_2(A_n)\) such that \(\tau^{l+1}U=U\). Then all indecomposable summands of \(U\) lie in the \(l\) outermost \(\tau\)-orbits of the AR-quiver of \(C_2(A_n)\). 
\end{lemma}

\begin{proof}
Assume that \(U_i\) is an indecomposable summand of \(U\), lying in \(\tau-orbit\) nr \(i\). Then for \(U_i\) we have that \(\Ext_\mathcal C(\tau^{-j}U_i,U_i)\neq0\) for \(j=0,\ldots, i\). If \(i>l\) this makes it impossible to find any other indecomposable summand \(\overline U_i\) of \(U\) such that \(\tau^{-l}U_i=\overline{U_i}\) . This is illustrated in figure \ref{ExtInL}.
\end{proof}

\begin{minipage}{\textwidth}
\[ \scalebox{.88} { \begin{tikzpicture}[scale=.5,yscale=-1]
	\fill [fill0] (2,4)--(3,3)--(2,2)--(1,3)--(2,4);
	\fill [fill1] (0,2)--(1,3)--(0,4)--(0,2);
	\fill [fill1] (3,3)--(5.5,0.5)--(6.5,.5)--(12,6)--(12,8)--(10.5,9.5)--(9.5,9.5)--(3,3);
 \foreach \x in {0,...,6}
  \foreach \y in {1,3,5,7,9}
   \node (\y-\x) at (\x*2,\y) [vertex] {};
 \foreach \x in {0,...,5}
  \foreach \y in {2,4,6,8}
   \node (\y-\x) at (\x*2+1,\y) [vertex] {};
 \replacevertex[fill0]{(3-1)}{[tvertex] {$U_{i}$}}
 \foreach \xa/\xb in {0/1,1/2,2/3,3/4,4/5,5/6}
  \foreach \ya/\yb in {1/2,3/4,5/4,5/6,7/6,9/8}
   {
    \draw [->] (\ya-\xa) -- (\yb-\xa);
    \draw [->] (\yb-\xa) -- (\ya-\xb);
   }
 \foreach \xa/\xb in {0/1,1/2,2/3,3/4,4/5,5/6}
  \foreach \ya/\yb in {3/2,7/8}
   {
    \draw [thick,loosely dotted] (\ya-\xa) -- (\yb-\xa);
    \draw [thick,loosely dotted] (\yb-\xa) -- (\ya-\xb);
   }
 \draw [dashed] (0,.5) -- (0,9); 
 \draw [dashed] (12,.5) -- (12,9);
\end{tikzpicture} } \]

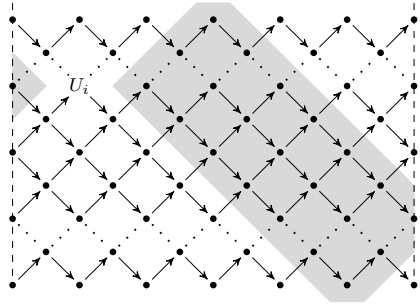
\captionof{figure}{The Ext-support of an indecomposable object in \(\mathcal C_2(A_n)\).} \label{ExtInL}
\end{minipage}

Note that all \(2\)-cluster tilting objects of type \(A_{3l}\) that do not contain a \(3\)-cycle contain indecomposable objects that do not lie in the \(l\)-outermost \(\tau\)-orbits. 

We now go on to show the last implication for the \(3\)-symmetric case, namely that if one has \(\tau^{l+1}T=T\) for a \(2\)-cluster tilting object \(T\) of \(\mathcal C_2(A_{3l})\), then \(T\) is \(3\)-symmetric. As remarked before theorem \ref{bowPerp}, there is a close relation between the placement of a vertex in the quiver a cluster tilted algebra, and the placement of the corresponding indecomposable vertex in the AR-quiver of the cluster category. Before proceeding with this last proof, we need to name a special type of vertices in the quiver, namely the maximal dividing vertices. Basically a maximal dividing vertex that divides the quiver in two subquivers of as equal size as possible. These will be the vertices that correspond to the indecomposables \(V_l\), \(V_{2l}\) and \(V_{3l}\) as in the proof of theorem \ref{3sumImp1} and in figure \ref{figure.oppe_og_nede}.

\begin{definition}
Let \(Q\) be the quiver of a \(2\)-cluster tilting object of type \(A_n\). We define a vertex \(M\) of \(Q\) to be a maximal dividing vertex if
\begin{itemize}
\item deleting \(M\) from \(Q\) yields to subquivers \(Q^M_1\) and \(Q^M_2\) that are each connected but not interconnected.
\item \(||Q^M_{1}|-|Q^M_{2}||=\min\left\{||Q^i_1|-|Q_2^i||\left|\right.i=1,\ldots,n\right\}\)
\end{itemize}
\end{definition}

\begin{theorem}\label{SecondImpl}
Let \(U\) be a \(2\)-cluster tilting object of \(\mathcal C_2(A_n)\) where \(n=3l\). If \(\tau^{l+1} U=U\) then \(U\) is \(3\)-symmetric.
\end{theorem}

\begin{proof}
Let \(Q\) be the quiver of \(\End_{\mathcal C}(U)^{op}\). We aim to show that \(U\) must have at least one indecomposable in the \(l\) outermost \(\tau\)-orbit.

Let \(i\) be a maximal dividing vertex of \(Q\), dividing the quiver into two quivers \(Q^i_1\) and \(Q^i_2\) with respectively \(n^i_1\) and \(n^i_2\) vertices. Assume that \(n^i_2\geq n^i_1\).  

We start by studying how \(Q^i_2\) is connected to the vertex \(i\) in \(Q\). There are two possibilities, either they are connected by a single arrow, or there is a \(3\)-cycle in \(Q\) formed by the vertex \(i\) and two vertices of \(Q^i_2\). 

First assume that \(i\) and \(Q^i_2\) are connected by a single arrow in \(Q\). Denote by \(i+1\) the vertex in \(Q^i_2\) that is a neighbour of \(i\) in \(Q\). Consider what happens if we instead divide the quiver \(Q\) at the vertex \(i+1\), this would give two subquivers of size \(n^i_2-1\) and \(n^i_1+1\). Assume that vertex \(i+1\) is also a maximal dividing vertex, and hence we obtain the same difference between subquivers by dividing at this vertex. Then \(n^i_2-n^i_1=|(n^i_2-1)-(n^i_1+1)|\), thus \(n^i_2=n^i_1+1\). Given any \(2\)-cluster tilting object \(T\) in \(\mathcal C_2(A_{3l})\) with quiver \(Q\), then by theorem \ref{bowPerp} the indecomposable object \(T_i\) corresponding to vertex \(i\) lies in the innermost \(\tau\)-orbit of the AR-quiver, contradicting theorem \ref{middleEmpty}. Let us therefore assume that the vertex \(i+1\) is not a maximal dividing vertex, giving rise to subquivers of size \(n^i_2-1\) and \(n^i_1+1\). Since \(i+1\) is not a maximal dividing vertex we must have the inequality \(n^i_2-n^i_1<|(n_2^i-1)-(n_1^i+1)|\), and so \(n_2^i=n_1^i\). This implies that any indecomposable object \(T_i\) corresponding to the vertex \(i\) in a \(2\)-cluster tilting object \(T\) with quiver \(Q\) must lie in the innermost \(\tau\)-orbit in the AR-quiver by theorem \ref{bowPerp}. This contradicts lemma \ref{middleEmpty}.

From the above considerations, it is clear that there is a \(3\)-cycle in \(Q\) formed by the vertex \(i\) and two vertices of \(Q_2^i\). Denote the other two vertices in the \(3\)-cycle by \(2i\) and \(3i\). Dividing at vertex \(2i\) gives rise to a subquiver \(Q_{2i}\) of size \(n^{2i}_2\) not containing \(i\) and \(3i\), and similarly dividing at vertex \(3i\) gives rise to a subquiver \(Q_{3i}\) of size \(n^{3i}_3\) not containing \(i\) or \(2i\). Hence \(n^i_1\geq\max\{n^{2i}_2,n^{3i}_3\}\) since \(i\) is a maximal dividing vertex. The situation at this point is illustrated in figure \ref{bevisIll}. Recall that the vertices \(i,2i\) and \(3i\) are not counted as vertices in the quivers \(Q^i_2,Q_{2i}\) and \(Q_{3i}\). Hence we have \(n^{i}_1+n^{2i}_2+n^{3i}_3+3=n\). Combined with the last inequality we then have \(3n^i_1+3\geq n\) so \(n^i_1\geq l-1\) (recall that n=3l). By theorem \ref{bowPerp} and lemma \ref{middleEmpty} this means that any indecomposable object \(U_i\) of \(U\) corresponding to vertex \(i\) must lie in the \(l^{th}\) outermost \(\tau\)-orbit of the AR-quiver of \(\mathcal C_2(A_{3l})\). However, by the assumption that \(\tau^{l+1}U=U\), there are at least two more indecomposable objects of \(U\) lying in this \(\tau\)-orbit, namely \(\tau^{l+1}U_i\) and \(\tau^{2l+2}U_i\). By the proof of lemma \ref{middleEmpty} this means that \(U\) is \(3\)-symmetric.  
\end{proof}

\begin{figure}
\begin{tikzpicture}
\node(P1) at (0,1.3335) [inner sep=1pt] {$\star$}; 
\node(P2) at (-1.155,-0.6665)[inner sep=1 pt] {$\star$};
\node(P3) at (1.155,-0.6665)[inner sep=1 pt]{$\star$};
\draw [->] (P1) -- (P2);
\draw [->](P2)--(P3);
\draw [->](P3)--(P1);
\draw [rotate=30](0,0) (-1.6,0) ellipse (24pt and 15 pt);
\node (S1) at (30: -1.8cm)[tvertex]{$Q_1$};
\draw [rotate=150](0,0) (-1.6,0) ellipse (24pt and 15 pt);
\node (S_2) at (150:-1.8 cm)[tvertex]{$Q_2$};
\draw [rotate=270](0,0) (-1.6,0) ellipse (24pt and 15 pt);
\node(S_3) at (90: 1.8 cm)[tvertex]{$Q_3$};
\end{tikzpicture}\caption{Illustration of proof of Theorem \ref{SecondImpl}}\label{bevisIll}
\end{figure}
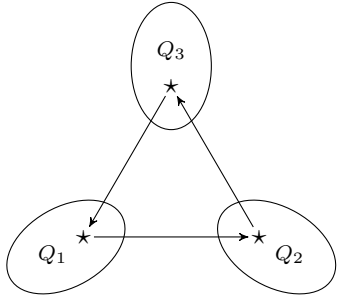

\begin{theorem}\label{type2A}
Let $T$ be a \(3\)-symmetric $2$-cluster tilting object of type $A_n$ and $\mathcal{X}$ be the preimage in the bounded derived category. Then \(\tau^{l+1}\mathcal X=\mathcal X\).
\end{theorem}

\begin{proof}
Follows from lemma \ref{periodInDerived}.
\end{proof}

\subsection{Summary of results type A}
We give a brief recount of the results obtained for type A.

\begin{theorem}\label{sumA}
Let \(T\) be a \(2\)-cluster tilting object of \(\mathcal C_2(A_n)\) such that \(\tau^sT=T\) for some \(0<s<n+3\). Then 
\begin{itemize}
\item \(s=\frac{n+3}{2}\) for odd \(n\) or
\item \(s=\frac{n+3}{3}\) for \(n\) divisible by \(3\).
\end{itemize}
 Furthermore
\begin{itemize}
\item \(s=\frac{n+3}{2}\) for odd \(n\) if and only if \(T\) is \(2\)-symmetric and 
\item  \(s=\frac{n+3}{3}\) for \(n\) divisible by \(3\) if and only if \(T\) is \(3\)-symmtric. 
\end{itemize}
\end{theorem}

Based on theorem \ref{sumA} we give the following corollary \ref{overviewA}, giving a complete overview of all triangulated orbit categories \(\mathcal O_F(A_n)\) containing a \(2\)-cluster tilting object.

\begin{corollary}\label{overviewA}
Let \(\mathcal O_F(A_n)=\mathcal D^b(A_n)/F\) for \(n\geq 1\). Then \(\mathcal O_F(A_n)\) has a \(2\)-cluster tilting object if and only if \(F\) is of the following form:
\begin{itemize}
\item \(F=\tau^{t(n+3)}\) for all \(n\geq 1\) and \(t\in\mathbb Z\)
\item \(F=\tau^{t(n+3)-1}\left[1\right]\) for all \(n\geq 1\) and \(t\in\mathbb Z\)
\item \(F=\tau^{t(l+2)}\) for \(n=2l+1\) and \(t\in\mathbb Z\)
\item \(F=\tau^{t(l+2)-1}\left[1\right]\) for \(n=2l+1\) and \(t\in\mathbb Z\)
\item \(F=\tau^{t(l+1)}\) for \(n=3l\) and \(t\in\mathbb Z\)
\item \(F=\tau^{t(l+1)-1}\left[1\right]\) for \(n=3l\) and \(t\in\mathbb Z\)
\end{itemize}.
\end{corollary}

\begin{proof}
There are two main types of AR-quivers to consider, cylindrical and Moebius-shaped. For each of these two main types we consider the three different periodicities a \(2\)-cluster tilting object of type \(A_n\) may have and apply lemma \ref{ImCltilt} and lemma \ref{periodInDerived}. 
\end{proof}


\section{$2$-cluster tilting subcategories of type $D$}\label{typeD}
In this section the focus will be on Dynkin diagrams of type $D$ and \(2\)-cluster tilting subcategories \(T\) of these. We will first discuss the AR-quiver, in order to obtain some value of \(s\) such that \(\tau^sT=T\). Recall from section \ref{quivers}, that quivers of cluster tilted algebras of type D are divided into three subcategories, depending on the number and distribution of \(\alpha\)-objects. We will treat each subtype of quiver in its own subsection.  

The AR-quiver of \(\mathcal D^b(D_n)\) is illustrated in figure \ref{AR-quiverD}. It is important to note that for even values of \(n\) the suspension functor in \(\mathcal D^b(D_n)\) sends an indecomposable object to an object in the same \(\tau\)-orbit. However for odd values of \(n\) the suspension functor sends an \(\alpha\)-object to an \(\alpha\)-object in the other \(\tau\)-orbit containing \(\alpha\)-objects. For indecomposable  \(\beta\)-objects the suspension functor sends the object to another object in the same \(\tau\)-orbit, both for even and odd values of \(n\). 

This impacts the number of objects in each \(\tau\)-orbit of the AR-quiver of \(\mathcal C_2(D_n)\). For even values of \(n\), all \(\tau\)-orbits of the AR-quiver of \(\mathcal C_2(D_n)\) contains \(n\) objects. However for odd values of \(n\), it is well-known that the top two rows of objects, the \(\alpha\)-objects, in the AR-quiver of \(\mathcal C_2(D_n)\) form a Moebius-band. Hence there is only one \(\tau\)-orbit of \(\alpha\)-objects, containing a total of \(2n\) objects. The \(\tau\)-orbits containing \(\beta\)-objects contain \(n\) objects also for odd values of \(n\). 

From the above remarks on the length of \(\tau\)-orbits in the AR-quiver of \(\mathcal C_2(D_n)\) it follows that for a \(2\)-cluster tilting object in \(\mathcal C_2(D_n)\) we have \(\tau^{2n}T=T\) for all \(n\). For even values of \(n\) we have \(\tau^nT=T\). However for any indecomposable \(\beta\)-object \(B\) we have \(\tau^nB=B\) in \(\mathcal C_2(D_n)\) for all values of \(n\).  

At this point one should recall that by theorem \ref{hermundD}, we can only expect to classify \(2\)-cluster tilting objects of type \(D_n\) up to flip of the \(\alpha\)-objects. We will see that for \(2\)-cluster tilting objects of type 1 this does not affect the results, however for type 2 and 3 this is of significance. 

Before proceeding with the three subtypes of quivers of type D, we have a result limiting possible values that \(s\) may take such that \(\tau^s T=T\):

\begin{lemma}\label{factor2}
If \(T\) is a \(2\)-cluster tilting object of \(\mathcal D_n\) such that the quiver \(Q\) is of type 2 or 3, and \(\tau^s T=T\), then \(s\) is an even number.
\end{lemma}

\begin{proof}
All the three subtypes of quivers have at least two vertices corresponding to \(\alpha\)-objects in the AR-quiver of \(\mathcal C_2(D_n)\). Hence this can be seen directly from the \(\Ext\)-support of an \(\alpha\)-object, illustrated in figure \ref{extSupportAlpha}.
\end{proof}

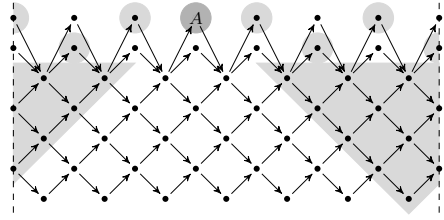
\begin{figure}[h]
\[ \scalebox{.8}{ \begin{tikzpicture}[scale=.5,yscale=-1]
 \foreach \x in {2,4,6}
  \fill [fill1] (\x*2,1) circle (.5);
 \fill [fill2] (6,1) circle (.5);
 \foreach \x in {1,5}
  \fill [fill1] (\x*2,2) circle (.5);
 \fill [fill1] (0,.5) arc (-90:90:.5);
 \fill [fill1] (14,2.5) arc (90:270:.5);
 \fill [fill1] (0,2.5) -- (4,2.5) -- (0,6.5) -- cycle;
 \fill [fill1] (8,2.5) -- (14,2.5) -- (14,6.5) -- (13,7.5) -- cycle;
 \foreach \x in {0,...,7}
 \foreach \y in {1,2,4,6}
 \node (\y-\x) at (\x*2,\y) [vertex] {};
 \foreach \x in {0,...,6}
 \foreach \y in {3,5,7}
   \node (\y-\x) at (\x*2+1,\y) [vertex] {};
 \replacevertex[fill2]{(1-3)}{[tvertex] {$A$}}
 \foreach \xa/\xb in {0/1,1/2,2/3,3/4,4/5,5/6,6/7}
 \foreach \ya/\yb in {1/3,2/3,4/3,4/5,6/5,6/7}
   {
    \draw [->] (\ya-\xa) -- (\yb-\xa);
    \draw [->] (\yb-\xa) -- (\ya-\xb);
   }
 \draw [dashed] (0,.5) -- (0,7.5); 
 \draw [dashed] (14,.5) -- (14,7.5);
\end{tikzpicture} } \]\caption{The indecomposables with a light grey background may not be part of the same \(2\)-cluster tilting object as \(A\).}\label{extSupportAlpha}
\end{figure}

\subsection{$2$-cluster tilting objects/subcategories of type $D$ subtype 1}

We are now ready to consider the first subtype of cluster tilting objects  of type $D$. Recall that a  \(2\)-cluster tilting object \(T\) giving rise to a cluster tilted algebra of type \(D_n\) of type 1 have exactly two indecomposable objects that are \(\alpha\)-objects, \(T_1\) and \(\phi T_1\). Knowing the exact number and placement of the \(\alpha\)-objects give us the following result:

\begin{theorem}
Let \(T\) be a \(2\)-cluster tilting object of \(\mathcal C_2(D_n)\) giving rise to a cluster tilted algebra \(\End_{\mathcal C}(T)^{op}\) with quiver \(Q\) of type 1. Assume that \(\tau^sT=T\) and \(s\) is the smallest integer such that this is true. Then \(s=n\).
\end{theorem}

\begin{proof}
First assume that \(n\) is an even number. Then there are two \(\tau\)-orbits of \(\alpha\)-objects, each containing \(n\) objects. Since \(T_1\) and \(\phi T_1\) do not lie in the same \(\tau\)-orbit, and these are the only \(\alpha\)-objects it is clear that \(n\) is the smallest value of \(s\) such that \(\tau^sT_1=T_1\).
Now assume that \(n\) is an odd number. Then \(T_1\) and \(\phi T_1\) lie in the same \(\tau\)-orbit containing \(2n\) objects. Hence \(\tau^nT_1=\phi T_1\) and \(\tau^n \phi T_1=T_1\), and therefore \(n\) is the smallest value of \(s\) such that \(\tau^sT=T\).
\end{proof}

\subsection{$2$-cluster tilting objects/subcategories of type $D$ subtype 2}

The second subtype of \(2\)-cluster tilting objects of \(\mathcal C_2(D_n)\) have two \(\alpha\)-objects \(T_1\) and \(T_2\), such that \(T_2\neq\phi T_1\). Furthermore from theorem \ref{alphaDistribution} we know that if the size of the subquivers of type \(A\) attached at the connecting vertices are \(n_1\) and \(n_2\) then \(T_1=\tau^{n_1+1}T_2\) or \(T_1=\tau^{n_1+1}\phi T_2\), and \(T_2=\tau^{n_2+1}T_1\) or \(T_2=\tau^{n_2+1}\phi T_1\).

\begin{definition}
A \(2\)-cluster tilting object of subtype 2 is called \(2\)-symmtric if \(Q\) has the same subquiver of type \(A\) attached at both connecting vertices.  
\end{definition}

\begin{theorem}
Let \(T\) be a \(2\)-cluster tilting object of type \(D_n\) of subtype 2. Assume that \(s\) is the smallest number such that \(\tau^sT=T\). Then
\begin{itemize}
\item [1.] for even values of \(n\) we have \(s=n/2\) if and only if \(T\) is \(2\)-symmetric and \(4|n\). Otherwise \(s=n\).
\item [2.] for odd \(n\) we have \(s=2n\) for all \(T\).
\end{itemize} 
\end{theorem}

\begin{proof}
Denote the \(\alpha\)-objects by \(T_1\) and \(T_2\).

First assume that \(T\) is \(2\)-symmetric and that \(4|n\). All \(\tau\)-orbits have \(n\) objects, so \(s\leq n\). The same subquiver of type \(A_{n/2-1}\) is attached at both connecting vertices. We note that \(n/2\) is an even number and thus theorem \ref{alphaDistribution}  implies that \(\tau^{n/2}T_1=T_2\) and \(\tau^{n/2}T_2=T_1\).  

Denote by \(\Delta_1\) and \(\Delta_2\) the subcategories of type \(\modf A_l\) attached at the connecting vertices of \(Q\). By lemma \ref{MorpAtriangle} there is a morphism from \(\Delta_1\) to \(T_1\) factoring through the projective injective object  \(\Pi(\Delta_1)\). Similarly there is a morphism from \(\Delta_2\) to \(T_2\) factoring through \(\Pi(\Delta_2)\). Hence for each pair of corresponding vertices \(T_{\Delta_1}\) and \(T_{\Delta_2}\) in \(\Delta_1\) and \(\Delta_2\) we have \(\tau^{n/2}T_{\Delta_1}=T_{\Delta_2}\) and \(\tau^{n/2}T_{\Delta_2}=T_{\Delta_1}\). Since it is the same subquiver of type \(A\) attached at both connecting vertices, we have \(\tau^{n/2}T=T\).

Now assume that \(T\) is a \(2\)-cluster tilting object of type \(D_n\) such that  \(\tau^{n/2}T=T\). This can only be the case if \(2|n\). From lemma \ref{factor2} it follows that \(s=n/2\) is an even number, hence \(4|n\). 

Let \(Q_1\) and \(Q_2\) be the subquivers of type \(A_{n/2-1}\) attached at the two connecting vertices. It is clear that \(Q_1\) and \(Q_2\) must be the same quiver in the mutation class of type \(A\). Hence \(T\) is \(2\)-symmtric.

If \(n\) is odd then the \(\tau\)-orbit containing \(\alpha\)-objects contains \(2n\) objects. The only possibility of achieving a better value of \(s\) than \(2n\), is that \(\tau^nT_1=T_2\) and \(\tau^nT_2=T_1\). This however would mean that \(T_1=\phi T_2\), a contradition of the choice of \(T\).  
\end{proof}


\subsection{$2$-cluster tilting objects/subcategories of type $D$ subtype 3}

We recall from subsection \ref{ClTiltAlgTypeD} that cluster tilting objects of subtype 3 have at least three indecomposable \(\alpha\)-objects. The distribution of these \(\alpha\)-objects are described in theorem \ref{alphaDistribution}, they form the central cycle of the quiver of \(\End_{\mathcal C_2(D_n)}(T)^{op}\). We start off by studying how the number of \(\alpha\)-objects is affected by \(T\) being closed under \(\tau^s\). Recall that if \(T\) is a \(2\)-cluster tilting object of \(\mathcal C_2(D_n)\) such that \(\tau^sT=T\), then \(s\) is an even number by lemma \ref{factor2}.

Since we are interested in studying the cases \(\tau^sT=T\) for \(s<n\) for even \(n\) and \(s<2n\) for odd values of \(n\), it is clear that if \(\tau^sT=T\), then \(s\) divides \(n\) for even \(n\) and \(s\) divides \(2n\) for odd \(n\).  

\begin{lemma}\label{alphaInMsym}
Let \(T\) be a \(2\)-cluster tiling object of \(\mathcal C_2(D_n)\) and \(s=2k\) an even number such tbhat \(\tau^sT=T\). Then the number of indecomposable \(\alpha\)-objects \(a\) in \(T\) is divisible by \(\frac{n}{s}\) for even \(n\) and is divisible by \(\frac{2n}{s}\) for odd \(n\). 
\end{lemma}

\begin{proof}
Assume that \(n\) is even. There are two \(\tau\)-orbits of length \(n\) containing \(\alpha\)-objects, each may contain multiple of \(\frac{n}{s}\).

Now assume that \(n\) is an odd number. Then there is only one \(\tau\)-orbit of \(\alpha\)-objects of length \(2n\). Clearly this \(\tau\)-orbit contains a multiple of \(\frac{2n}{s}\) indecomposables contained in \(T\). 
\end{proof}

\begin{figure}[h!]
\begin{tikzpicture}[font=\small, scale=1.1]
\node[rotate around={90:(0,-2.5)}](atm-1) at (0,0)[inner sep=1pt]{\rotatebox{-90}{\(v_{a-1}\)}};
\node[rotate around={54:(0,-2.5)}](atm) at (0,0)[inner sep=1pt]{\rotatebox{-54}{\(v_{a}\)}};
\node[rotate around={18:(0,-2.5)}](a1) at (0,0)[inner sep=1pt]{\rotatebox{-18}{\(v_1\)}};
\node[rotate around={-18:(0,-2.5)}](a2) at (0,0)[inner sep=1pt]{\rotatebox{18}{\(v_2\)}};
\node[rotate around={-54:(0,-2.5)}](a3) at (0,0)[inner sep=1pt]{\rotatebox{54}{\(v_3\)}};
\node[rotate around={-90:(0,-2.5)}](a4) at (0,0)[inner sep=1pt]{\rotatebox{90}{\(v_{i-1}\)}};
\node[rotate around={-126:(0,-2.5)}](a5) at (0,0)[inner sep=1pt]{\rotatebox{126}{\(v_{i}\)}};
\node[rotate around={-162:(0,-2.5)}](a6) at (0,0)[inner sep=1pt]{\rotatebox{162}{\(v_{i+1}\)}};
\node[rotate around={-198:(0,-2.5)}](a7) at (0,0)[inner sep=1pt]{\rotatebox{198}{\(v_{i+2}\)}};
\node[rotate around={-234:(0,-2.5)}](a8) at (0,0)[inner sep=1pt]{\rotatebox{234}{\(v_{i+3}\)}};

\node[rotate around={0:(0,-3.5)}](n1) at (0,1)[inner sep=2pt]{\(Q_1\)};
\node[rotate around={-36:(0,-3.5)}](n2) at (0,1)[inner sep=2pt]{\rotatebox{36}{\(Q_2\)}};
\node[rotate around={-72:(0,-3.5)}](n3) at (0,1){};
\node[rotate around={-108:(0,-3.5)}](n4) at (0,1)[inner sep=2pt]{\rotatebox{108}{\(Q_{i-1}\)}};
\node[rotate around={-144:(0,-3.5)}](n5) at (0,1)[inner sep=2pt]{\rotatebox{144}{\(Q_i\)}};
\node[rotate around={-180:(0,-3.5)}](n6) at (0,1)[inner sep=2pt]{\rotatebox{180}{\(Q_{i+1}\)}};
\node[rotate around={-216:(0,-3.5)}](n9) at (0,1)[inner sep=1pt]{\rotatebox{216}{\(Q_{i+2}\)}};
\node[rotate around={36:(0,-3.5)}](n7) at (0,1)[inner sep=2pt]{\rotatebox{-36}{\(Q_a\)}};
\node[rotate around={72:(0,-3.5)}](n8) at (0,1)[inner sep=2pt]{\rotatebox{-72}{\(Q_{a-1}\)}};

\draw[->] (a1) -- (a2);
\draw[->] (a2) -- (a3);
\draw[->] (atm) -- (a1);
\draw[->] (a4) -- (a5);
\draw[->] (a5) -- (a6);
\draw[->] (a6) -- (a7);
\draw[->] (atm-1)--(atm);
\draw[->] (a7)--(a8);
\draw[dotted] (a3) --(a4);
\draw[dotted] (a8)--(atm-1);

\draw[->] (a2) --(n1);
\draw[->] (n1)--(a1);
\draw[->] (a3)--(n2);
\draw[->](n2)--(a2);
\draw[->](a5)--(n4);
\draw[->](n4)--(a4);
\draw[->](a6)--(n5);
\draw[->](n5)--(a5);
\draw[->](a7)--(n6);
\draw[->](n6)--(a6);
\draw[->](a8)--(n9);
\draw[->](n9)--(a7);
\draw[->](atm)--(n8);
\draw[->](n8)--(atm-1);
\draw[->](a1)--(n7);
\draw[->](n7)--(atm);
\end{tikzpicture}\caption{The quiver \(Q\) of \(\End_{\mathcal C_2}(T)^{op}\) of a \(2\)-cluster tilting object \(T\) of \(\mathcal C_2(D_n)\). Each \(Q_i\) represents the quiver of a cluster tilting object of type \(A\). }\label{typeDmsymmetric2}
\end{figure}
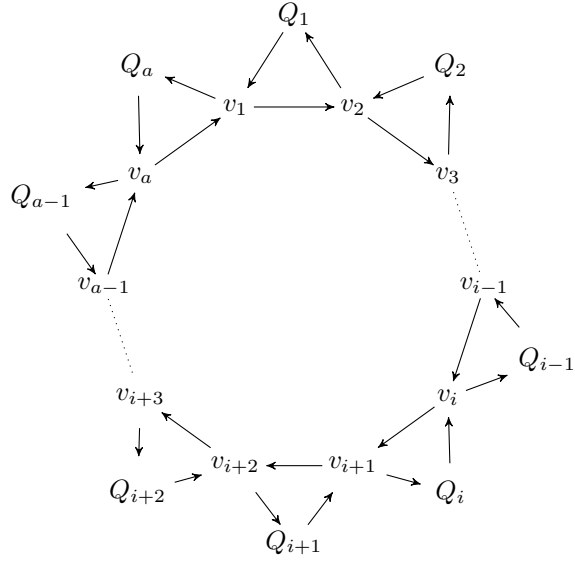

As for type A, there are quivers of type D which are particularly nice when one looks for symmetric properties. We now give a definition, describing the details of such quivers of type D. 

\begin{definition}\label{msym}
Let \(T\) be a \(2\)-cluster tiling object of \(\mathcal C_2(D_n)\) of type 3, with quiver as illustrated in figure \ref{typeDmsymmetric2}, where \(a\) is the number of \(\alpha\)-objects. If \(l\leq a\) is the largest integer such that \(a=l\cdot t\) and  \(Q_{r}=Q_{r+t}=Q_{r+2t}=\ldots=Q_{r+(l-1)t}\) for each \(r\in\left\{1,\ldots,t\right\}\) then we call \(T\) \(l\)-symmetric. Denote the number of vertices in subquiver \(Q_i\) by \(n_i\). If \(T\) is an \(l\)-symmetric cluster tilting object with quiver \(Q\), then we define \(p:=\sum_{i=1}^tn_i+t\).
\end{definition}

For a \(2\)-cluster tilitng object of \(\mathcal C_2(D_n)\) that is \(l\)-symmetric we can now see how \(l\) relates to \(n\):

\begin{lemma} \label{lRelatesTon}
Let \(\mathcal T\) be a \(2\)-cluster tilting object of \(\mathcal C_2(D_n)\) that is \(l\)-symmetric. Then we have that \(p\) divides \(n\).
\end{lemma}

\begin{proof}
Follows directly from definition \ref{msym}.
\end{proof}

Before proceeding with the main results of this section we need one more technical definition:
\begin{definition}
Let \(T\) be a \(2\)-cluster tilting object, with quiver \(Q\), such that \(\tau^sT=T\). Denote the subquivers attached to the central cycle by \(Q_i\) for \(i\in\left\{1,\ldots,mt\right\}\). We then define an equivalence relation on the set of quivers \(Q_i\) by defining that \(Q_i\sim_{s} Q_j\) if \(\tau^{sw}T_{Q_i}=T_{Q_j}\) for some \(w\in\mathbb Z\). We denote the equivalence class of a quiver \(Q_j\) by \(\left[Q_i\right]\), where \(i\) is the smallest index of any quiver in the equivalence class. Denote by \(n_i\) the number of vertices in the quivers in the equivalence class \(\left[Q_i\right]\).
\end{definition}

This will help us distinguish quivers attached to the central cycle that are isomorphic as quivers, but where the corresponding summands are not closed under \(\tau^s\).

\begin{theorem}\label{mainD}
Let \(T\) be a \(2\)-cluster tiling object of \(\mathcal C_2(D_n)\) with corresponding quiver of subtype 3. Assume that \(T\) is \(l\)-symmetric and that \(s\) is the smallest positive integer such that \(\tau^sT=T\). Then \(s=\begin{cases}  n/l&\mbox{if } n/l \mbox{ is even} \\
									2n/l & \mbox{ if } n/l \mbox{ is odd.}
\end{cases}\)
\end{theorem}

\begin{proof}
First assume that \(\tilde{s}=\begin{cases} n/l&\mbox{if } n/l \mbox{ is even} \\
									2n/l & \mbox{ if } n/l \mbox{ is odd.} \end{cases}\), hence \(\tilde s\) is an even number. Recall that \(T\) is \(l\)-symmetric, so we have \(n=lp\). Hence if \(n/l\) is even we have \(\tilde{s}=p\), and if \(n/l\) is odd we have \(\tilde{s}=2p\). We want to show that \(\tau^{\tilde{s}}T=T\). By theorem \ref{alphaDistribution} we know that \(v_{i+1}=\tau^{n_{i}+1}v_i\) if \(n_{i}+1\) is even  or \(v_{i+1}=\phi\tau^{n_{i}+1}v_i\) if \(n_{i}+1\) is odd. Hence it follows in the case of \(n/l\) being even, from theorem \ref{alphaDistribution}, that \(v_{i+t}=\tau^p v_i\), and  \(v_{mt+j}=\tau^p v_{(m+1)t+j}\) for each \(j\in\{1,\ldots,t\}\) and each \(m\in\{1,\ldots,l\}\). Furthermore, in the case of \(n/l\) being odd,  we have \(v_{i+2t}=\tau^{2p}v_i\) and \(v_{mt+j}=\tau^{2p}v_{(m+2)t+j}\) for each \(j\in\left\{1,\ldots,t\right\}\) and each \(l\in\left\{1,\ldots,l\right\}\). Hence the set of \(\alpha\)-objects of \(T\) are closed under \(\tau^{\tilde s}\) for both even and odd values of \(n/l\).  

Let \(T_{Q_i}\) be all indecomposable summands of \(T\) corresponding to all vertices in all the \(t\) occurences of subquiver \(Q_i\) in \(Q\). By lemma \ref{MorpAtriangle} it now follows that \(\tau^{\tilde s}T_{Q_i}=T_{Q_i}\), hence we have \(\tau^{\tilde s} T=T\). It follows that since \(s\) is minimal such that \(\tau^sT=T\), we have \(s | \tilde s\).

Now assume that \(\tilde l=\begin{cases}n/s & \text{if \(n\) is even}  \\
 2n/s & \text{ if \(n\) is odd}\end{cases}\). Note that this means that \(\tilde l\) is an integer, as \(\tau^sT=T\) with \(s\) an even number and the length of \(\tau\)-orbits is either \(n\) or \(2n\). We want to show that the quiver of \(T\) satisfies all properties of being \(\tilde l\)-symmetric, with the exception that we do not require \(\tilde l\) being the maximal value such that this happens. That is, we want to show that the quiver of \(T\) has a central cycle corresponding to \(a=\tilde{l}\cdot t\) number of \(\alpha\)-objects with the following relations on the attached quivers:
 \(Q_{r}=Q_{r+t}=Q_{r+2t}=\ldots=Q_{r+(\tilde{l}-1)t}\) for each \(r\in\left\{1,\ldots,t\right\}\).

From lemma \ref{alphaInMsym} it follows that the quiver \(Q\) of \(\End_{\mathcal C}(T)^{op}\) has the shape of an \(\tilde lt\)-cycle with subquivers from the mutation class of type \(A\) attached at the connecting vertices. Let \(Q_i\) be a subquiver attached at a connecting vertex of \(Q\). Since \(\tau^sT=T\) the same quiver \(Q_i\) occurs at a total of \(\tilde l\) connecting vertices. Denote by \(n_i\) the number of vertices in quiver \(Q_i\).

The total number of vertices in \(Q\) is \(n\), so from the above considerations we get the following equation: \[\tilde lt+\tilde l\Sigma_{i=1}^tn_i=\tilde l (t+\Sigma_{i=1}^t n_i)=n.\] From which it follows that \( t+\Sigma_{i=1}^t n_i=\begin{cases}s & \text{ if \(n\) is even } \\
						s/2 & \text{ if \(n\) is odd}  \end{cases}\). 

We now want to show that the distribution of the quivers \(Q_i\) amongst the connecting vertices is as in the definiton of \(\tilde l\)-symmetric. Vertices in the central \(\tilde lt\)-cycle of \(Q\) will be denoted with the same notation as in definiton \ref{msym}.

Assume first that \(n\) is even, hence \(\tilde l =n/s\) and \(t+\Sigma_{i=1}^t n_i= s\). We start with the vertices \(v_1\) and \(v_2\) in the central \(\tilde lt\)-cycle. Attached at these vertices is a quiver from the equivalence class \(\left[Q_1\right]\). Next, at the vertices \(v_2\) and \(v_3\) there is a quiver attached from the equivalence class \(\left[Q_2\right]\). Continuing like this, at vertices \(v_j\) and \(v_{j+1}\) there is a quiver attached from the equivalence class \(\left[Q_j\right]\) for each \(j\in\left\{1,\ldots,t\right\}\). Therefore we have \(\tau^sv_1=v_{t+1}\), since \(s\) is even. By lemma \ref{MorpAtriangle} it follows that the quiver attached at vertices \(v_{t+1}\) and \(v_{t+2}\) is a quiver from the equivalence class \(\left[Q_2\right]\). By iterating this process we find that \(Q_r=Q_{r+t}=\ldots=Q_{r+(\tilde l-1)t}\) for \(r\in\left\{1,\ldots,t\right\}\).

Now assume that \(n\) is odd, hence \(\tilde l=2n/s\) and \(t+\Sigma_{i=1}^tn_i=s/2\). Notice that \(n=\tilde ls/2\), so since \(n\) is odd both \(s/2\) and \(\tilde l\) are odd integers. Again in this case we start with vertices \(v_1\) and \(v_2\) in the central \(\tilde lt\)-cycle. Attached at these vertices is a quiver from the equivalence class \(\left[Q_1\right]\). Next, at vertices \(v_2\) and \(v_3\) there is a quiver attached from the equivalence class \(\left[Q_2\right]\). Continuing like this, at vertices \(v_j\) and \(v_{j+1}\) there is a quiver attached from the equivalence class \(\left[Q_j\right]\). Thus we find that \(\phi\tau^{s/2}v_1=v_{1+t}\). We now want to show that the quiver \(Q_x\) attached at the pair of vertices \(v_{1+t}\) and \(v_{2+t}\) is in the equivalence class of \(\left[Q_1\right]\). Considering that \(\phi=\tau^n\) for odd \(n\) we find that \[\phi\tau^{s/2}v_1=\tau^n\tau^{s/2}v_1=\tau^{\tilde ls/2+s/2}v_1=\tau^{s(\tilde l/2+1/2)}v_1=\tau^{sw}v_1=v_{t+1}\] where \(w=\tilde l/2+1/2\) is an integer since \(\tilde l\) is odd. Hence the quiver attached at vertices \(v_{t+1}\) and \(v_{t+2}\) is from the equivalence class \(\left[Q_1\right]\). By iterating this process we find that \(Q_r=Q_{r+t}=\ldots=Q_{r+(\tilde l-1)t}\) for \(r\in\left\{1,\ldots,t\right\}\).

From this it follows that \(\tilde l|l\), hence \(l=\alpha\tilde l\). Also recall that \newline \(s|\tilde s=\begin{cases}n/l& \text{ if \(n/l\) is even}\\
															2n/l&\text{ if \(n/l\) is odd}\end{cases}\).

We now again start by considering the case of \(n\) an even number. From the definition of \(\tilde l\) we then find that \(s=n/\tilde l=\alpha n/l\). If the quotient \(n/l\) is even then since \(s|\tilde s\) we have \(\dfrac{\alpha n}{l}\left|\right.\dfrac{n}{l}\) yielding \(\alpha=1\) and \(n=sl\). If however the quotient \(n/l\) we find that since \(s|\tilde s\) then \(\dfrac{\alpha n}{l}\left|\right.\dfrac{2n}{l}\) yielding \(\alpha=2\) and \(2n=sl\).

Now consider the case of \(n\) being an odd number, in which case we have \(\tilde l=2n/s\). Then the quotient \(n/l\) can not be an even number, hence there is only one case to consider which is \(n/l\) an odd number. In that case we have as above that \(s|\tilde s\) and hence \(\dfrac{2\alpha n}{l}\left|\right.\dfrac{2n}{l}\) yielding again \(\alpha=1\) and \(2n=sl\) which concludes the proof. 
\end{proof}

\subsection{Summary of result of type \(D_n\)}

We give a brief recount of the result obtained for type \(D_n\).

\begin{theorem}\label{sumD}
Let \(T\) be a \(2\)-cluster tilting object of \(\mathcal C_2(D_n)\) and \(s\) the smallest positive integer such that \(\tau^sT=T\). Then 
\begin{itemize}
\item If \(T\) is of subtype 1 then \(s=n\). 
\item If \(n\) is even and \(T\) is of subtype 2 then \(s=n/2\) if and only if \(T\) is \(2\)-symmetric of subtype 2 and \(4|n\). For all other cases where \(n\) is even we have \(s=n\).
\item If \(n\) is odd and \(T\) is of subtype 2 then \(s=2n\) for all \(T\). 
\item If \(T\) is of subtype \(3\) and \(l\) symmetric then \(s=\begin{cases} n/l& \text{ for \(n/l\) even }\\
2n/l & \text{for \(n/l\) odd }\end{cases}\).
\end{itemize}
\end{theorem}

Based on theorem \ref{sumD} we are able to give the following corollary, giving an overview of all triangulated orbit categories \(\mathcal O_F(D_n)\) containing a \(2\)-cluster tilting object.

\begin{corollary}\label{overviewD}
Let \(\mathcal O_F(D_n)=\mathcal D^b(D_n)/F\) be a triangulated orbit category with \(n\geq 4\). Then \(\mathcal O_F(D_n)\) has a \(2\)-cluster tilting object if and only if \(F\) is one of the following functors:
\begin{itemize}
\item \(\tau^{tn}\) for odd \(n\geq 5\) and \(t\in\mathbb Z\)
\item \(\tau^{tn-1}\left[1\right]\) for odd \(n\geq 5\) and \(t\in\mathbb Z\)
\item \(\tau^{2t}\) for all \(n\geq4\) and \(t\in\mathbb Z\).
\item \(\tau^{2t-1}\left[1\right]\) for odd \(n\geq 5\) and \(t\in\mathbb Z\).
\end{itemize}
\end{corollary}

\begin{proof}
For even values of \(n\) there is only one type of AR-quiver of \(\mathcal O_F\) to consider, which is cylindrical. However for odd \(n\) there are two possibilities, one is cylindrical and the other is cylindrical for \(\beta\)-objects with a Moebius-twist in the top two \(\tau\)-orbits consisting of the \(\alpha\)-objects. 

From the last part of theorem \ref{sumD} it follows that for both even and odd \(n\) the \(2\)-cluster tilting object where the corresponding quiver of \(\End_{\mathcal C}(T)^{op}\) is an \(n\)-cycle is periodic under \(\tau^2\). Since \(2\) is the smallest even natural number we get the last two parts of the corollary by applying lemma \ref{ImCltilt} and lemma \ref{periodInDerived}. These cases also cover all categories having a \(2\)-cluster tilting object due to the second part of theorem \ref{sumD}. 

The last part to consider is the first part of theorem \ref{sumD}, which gives rise to the first two parts of this corollary using lemma \ref{ImCltilt} and lemma \ref{periodInDerived}. It is only necessary to consider the cases for odd \(n\) as all cases of even \(n\) is considered in the third part of the corollary. 

\end{proof}


\section{$2$-cluster tilting subcategories of type $E$}\label{sectE}
In this section we focus on the cases, \(E_6, E_7\) and \(E_8\). We treat each case in its separate subsection . In each case we will use the properties of the AR-quiver of the \(2\)-cluster category to study the periodic properties, as in the two previous sections. 

\subsection{Type \(E_6\)}
The AR-quiver of \(\mathcal C_2(E_6)\) is illustrated in figure \ref{ARquiverE6}. It has a Moebius shape, the two outermost \(\tau\)-orbits each contain 14 objects, whereas the two innermost \(\tau\)-orbits contains only \(7\) objects each. Hence for any object \(X\) of \(\mathcal C_2(E_6)\) we have that \(\tau^{14}X=X\). From this information we deduce that other possible values of \(s\) such that \(\tau^s\mathcal T=\mathcal T\) for a \(2\)-cluster tilting subcategory of \(\mathcal C_2(E_6)\) are \(s=2\) or \(s=7\).

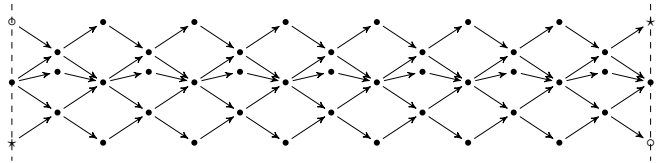
\begin{figure}[h]
\[ \scalebox{.8}{ \begin{tikzpicture}[scale=.5,yscale=-1]
 \foreach \x in {0,...,7}
  \foreach \y in {1,3,5}
   \node (\y-\x) at (\x*3,\y) [vertex] {};
 \foreach \x in {0,...,6}
  \foreach \y in {2,4}
   \node (\y-\x) at (\x*3+1.5,\y) [vertex] {};
 \foreach \x in {0,...,6}
  \foreach \y in {3}
   \node (a-\x) at (\x*3+1.5,\y-0.35)[vertex] {};
  
  \replacevertex{(1-0)}{[tvertex] {\(\circ\)}}
  \replacevertex{(5-7)}{[tvertex] {\(\circ\)}} 
  \replacevertex{(5-0)}{[tvertex]{\(\star\)}}
  \replacevertex{(1-7)}{[tvertex]{\(\star\)}}
 \foreach \xa/\xb in {0/1,1/2,2/3,3/4,4/5,5/6,6/7}
  \foreach \ya/\yb in {1/2,3/2,3/4,5/4}
   {
    \draw [->] (\ya-\xa) -- (\yb-\xa);
    \draw [->] (\yb-\xa) -- (\ya-\xb);
   }
 \foreach \xa/\xb in {0/1,1/2,2/3,3/4,4/5,5/6,6/7}
  \foreach \ya/\yb in {3/3}
  {
   \draw [->] (\ya-\xa) -- (a-\xa);
   \draw [->]  (a-\xa) -- (\yb-\xb);
  }
 \draw [dashed] (0,0.4) -- (0,5.6); 
 \draw [dashed] (21,0.4) -- (21,5.6);
\end{tikzpicture} } \]\caption{The AR-quiver of \(\mathcal C_2(E_6)\), the objects of same shape on either side are identified.}\label{ARquiverE6}
\end{figure}

Assume first that \(s=2\). If there is a summand of \(\mathcal T\) in any given \(\tau\)-orbit, then this \(\tau\)-orbit contains either 7 or 14 indecomposables from \(\mathcal T\). This is not possible as \(\mathcal T\) only has 6 indecomposable summands up to isomorphism.

For the case \(s=7\), each \(\tau\)-orbit of length 14 must contain an even number of summands from \(\mathcal T\). If a \(\tau\)-orbit has length 7 it may contain an odd number of indecomposable summands from \(\mathcal T\). Below we have included an illustration for each \(\tau\)-orbit of the AR-quiver of \(\mathcal C_2(E_6)\). Each diagram shows in grey the indecomposable objects that can not be part of the same \(2\)-cluster tilting object as \(M\). 

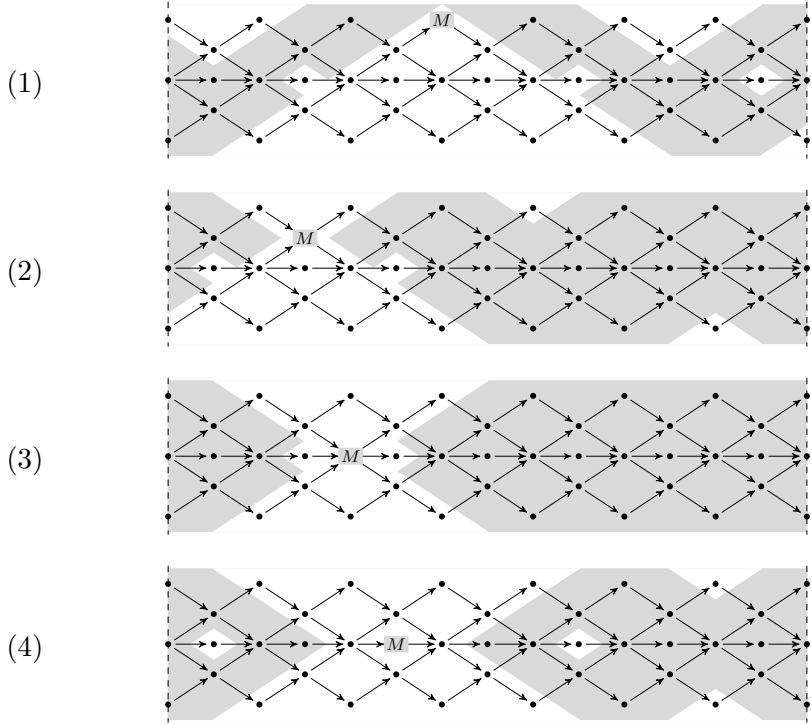
\begin{equation}
\label{E6_P1}
\scalebox{.8}{ \begin{tikzpicture}[baseline=-1.7cm,scale=.5,yscale=-1] 
\fill [fill1] (0,0.5)-- (21,0.5)--(21, 5.5)--(0,5.5)--cycle;

\fill [fill0] (9,0.5)--(10.5,1.5)--(12,2.5)--(13.5,3.5)--(15,4.5)--(16.5,5.5)--(1.5,5.5)--cycle;
\fill[fill0] (0,0.5)--(4.5,0.5)--(1.5,2.5)--(0,1.5)--cycle;
\fill[fill0] (13.5,0.5)--(19.5,0.5)--(16.5,2.5)--cycle;
\fill[fill0] (18.75,3)--(19.5,2.5)--(20.25,3)--(19.5,3.5)--cycle;
\fill[fill0] (12.75,3)--(13.5,2.5)--(14.25,3)--(13.5,3.5)--cycle;
\fill[fill0] (3.75,3)--(4.5,2.5)--(5.25,3)--(4.5,3.5)--cycle;
\fill[fill0] (19.5,5.5)--(21,4.5)--(21,5.5)--cycle;
 \foreach \x in {0,...,7}
  \foreach \y in {1,3,5}
   \node (\y-\x) at (\x*3,\y) [vertex] {};
 \foreach \x in {0,...,6}
  \foreach \y in {2,4}
   \node (\y-\x) at (\x*3+1.5,\y) [vertex] {};
 \foreach \x in {0,...,6}
  \foreach \y in {3}
   \node (a-\x) at (\x*3+1.5,\y)[vertex] {};
  
   \replacevertex{(1-3)}{[tvertex][fill1] {$M$}}  
 \foreach \xa/\xb in {0/1,1/2,2/3,3/4,4/5,5/6,6/7}
  \foreach \ya/\yb in {1/2,3/2,3/4,5/4}
   {
    \draw [->] (\ya-\xa) -- (\yb-\xa);
    \draw [->] (\yb-\xa) -- (\ya-\xb);
   }
 \foreach \xa/\xb in {0/1,1/2,2/3,3/4,4/5,5/6,6/7}
  \foreach \ya/\yb in {3/3}
  {
   \draw [->] (\ya-\xa) -- (a-\xa);
   \draw [->]  (a-\xa) -- (\yb-\xb);
  }
 \draw [dashed] (0,0.4) -- (0,5.6); 
 \draw [dashed] (21,0.4) -- (21,5.6);
\end{tikzpicture} } 
 \end{equation}

%
%
\begin{equation}
\label{E6_P2}
\scalebox{.8}{ \begin{tikzpicture}[baseline=-1.7cm,scale=.5,yscale=-1]
\fill [fill1] (0,0.5)-- (21,0.5)--(21, 5.5)--(0,5.5)--cycle;
\fill [fill0] (4.5,1.5)--(10.5,5.5)--(0,5.5)--(0,4.5)--cycle;
\fill[fill0](1.5,0.5)--(7.5,0.5)--(4.5,2.5)--cycle;
\fill[fill0](0.75,3)--(1.5,2.5)--(2.25,3)--(1.5,3.5)--cycle;
\fill[fill0](6.75,3)--(7.5,2.5)--(8.25,3)--(7.5,3.5)--cycle;
\fill[fill0](12,1.5)--(13.5,0.5)--(10.5,0.5)--cycle;
\fill[fill0](16.5,5.5)--(18,4.5)--(19.5,5.5)--cycle;
 \foreach \x in {0,...,7}
  \foreach \y in {1,3,5}
   \node (\y-\x) at (\x*3,\y) [vertex] {};
 \foreach \x in {0,...,6}
  \foreach \y in {2,4}
   \node (\y-\x) at (\x*3+1.5,\y) [vertex] {};
 \foreach \x in {0,...,6}
  \foreach \y in {3}
   \node (a-\x) at (\x*3+1.5,\y)[vertex] {};
   
   \replacevertex{(2-1)}{[tvertex][fill1] {$M$}} 
 \foreach \xa/\xb in {0/1,1/2,2/3,3/4,4/5,5/6,6/7}
  \foreach \ya/\yb in {1/2,3/2,3/4,5/4}
   {
    \draw [->] (\ya-\xa) -- (\yb-\xa);
    \draw [->] (\yb-\xa) -- (\ya-\xb);
   }
 \foreach \xa/\xb in {0/1,1/2,2/3,3/4,4/5,5/6,6/7}
  \foreach \ya/\yb in {3/3}
  {
   \draw [->] (\ya-\xa) -- (a-\xa);
   \draw [->]  (a-\xa) -- (\yb-\xb);
  }
 \draw [dashed] (0,0.4) -- (0,5.6); 
 \draw [dashed] (21,0.4) -- (21,5.6);
\end{tikzpicture} } 
\end{equation}
%

%
\begin{equation}
\label{E6_P3}
\scalebox{.8}{ \begin{tikzpicture}[baseline=-1.7cm,scale=.5,yscale=-1]
\fill [fill1] (0,0.5)-- (21,0.5)--(21, 5.5)--(0,5.5)--cycle;
\fill [fill0] (1.5,0.5)--(10.5,0.5)--(6,3.5)--cycle;
\fill[fill0](1.5,5.5)--(6,2.5)--(10.5,5.5)--cycle;
\fill[fill0](6.75,3)--(7.5,2.5)--(8.25,3)--(7.5,3.5)--cycle;
\fill[fill0](3.75,3)--(4.5,2.5)--(5.25,3)--(4.5,3.5)--cycle;

 \foreach \x in {0,...,7}
  \foreach \y in {1,3,5}
   \node (\y-\x) at (\x*3,\y) [vertex] {};
 \foreach \x in {0,...,6}
  \foreach \y in {2,4}
   \node (\y-\x) at (\x*3+1.5,\y) [vertex] {};
 \foreach \x in {0,...,6}
  \foreach \y in {3}
   \node (a-\x) at (\x*3+1.5,\y)[vertex] {};
   
   \replacevertex{(3-2)}{[tvertex][fill1] {$M$}} 
 \foreach \xa/\xb in {0/1,1/2,2/3,3/4,4/5,5/6,6/7}
  \foreach \ya/\yb in {1/2,3/2,3/4,5/4}
   {
    \draw [->] (\ya-\xa) -- (\yb-\xa);
    \draw [->] (\yb-\xa) -- (\ya-\xb);
   }
 \foreach \xa/\xb in {0/1,1/2,2/3,3/4,4/5,5/6,6/7}
  \foreach \ya/\yb in {3/3}
  {
   \draw [->] (\ya-\xa) -- (a-\xa);
   \draw [->]  (a-\xa) -- (\yb-\xb);
  }
 \draw [dashed] (0,0.4) -- (0,5.6); 
 \draw [dashed] (21,0.4) -- (21,5.6);
\end{tikzpicture} }
\end{equation}
%

%
\begin{equation}
\label{E6_P7}
 \scalebox{.8}{ \begin{tikzpicture}[baseline=-1.7cm,scale=.5,yscale=-1]
\fill [fill1] (0,0.5)-- (21,0.5)--(21, 5.5)--(0,5.5)--cycle;
\fill [fill0] (1.5,0.5)--(13.5,0.5)--(9.75,3)--(13.5,5.5)--(1.5,5.5)--(5.25,3)--cycle;
\fill[fill0](0.75,3)--(1.5,2.5)--(2.25,3)--(1.5,3.5)--cycle;
\fill[fill0](12.75,3)--(13.5,2.5)--(14.25,3)--(13.5,3.5)--cycle;
\fill[fill0](16.5,0.5)--(18,1.5)--(19.5,0.5)--cycle;
\fill[fill0](16.5,5.5)--(18,4.5)--(19.5,5.5)--cycle;
 \foreach \x in {0,...,7}
  \foreach \y in {1,3,5}
   \node (\y-\x) at (\x*3,\y) [vertex] {};
 \foreach \x in {0,...,6}
  \foreach \y in {2,4}
   \node (\y-\x) at (\x*3+1.5,\y) [vertex] {};
 \foreach \x in {0,...,6}
  \foreach \y in {3}
   \node (a-\x) at (\x*3+1.5,\y)[vertex] {};
   
   \replacevertex{(a-2)}{[tvertex][fill1] {$M$}} 
 \foreach \xa/\xb in {0/1,1/2,2/3,3/4,4/5,5/6,6/7}
  \foreach \ya/\yb in {1/2,3/2,3/4,5/4}
   {
    \draw [->] (\ya-\xa) -- (\yb-\xa);
    \draw [->] (\yb-\xa) -- (\ya-\xb);
   }
 \foreach \xa/\xb in {0/1,1/2,2/3,3/4,4/5,5/6,6/7}
  \foreach \ya/\yb in {3/3}
  {
   \draw [->] (\ya-\xa) -- (a-\xa);
   \draw [->]  (a-\xa) -- (\yb-\xb);
  }
 \draw [dashed] (0,0.4) -- (0,5.6); 
 \draw [dashed] (21,0.4) -- (21,5.6);
\end{tikzpicture} } 
\end{equation}
%

Based on these overviews for each \(\tau\)-orbit, one may use combinations of them together with combinatorial arguments to get a complete classification of all \(2\)-cluster tilting objects of \(\mathcal C_2(E_6)\) that are periodic under \(\tau^7\). An example of a calculation is given in subsection \ref{subsecE8}. 

Before looking at the classification itself, we need a notation to keep track of the distribution of the indecomposables in the AR-quiver. For a \(2\)-cluster tilting object \(T\) of \(\mathcal C_2(E_6)\) that is closed under \(\tau^7\),  we will associate a vector \((a_1,a_2,a_3,a_4)\). The number \(a_1\) will denote the number of indecomposables in the outermost \(\tau\)-orbit. The value of \(a_1\) will be denoted by  \(2\) if there are one pair of indecomposables summands in the outermost \(\tau\)-orbit, or \(2+2\) if there are two pairs. The complete classification is listed in table \ref{tableE6}.


\begin{table}[h]
\centering
\begin{tabular}{|m{2.2cm}|m{1.5cm}|m{9cm}|}
\hline
Distribution in AR-q. & \(\tau^sT=T\) & Corresponding quivers of \(T\) \\ \hline
(2,2,1,1) & s=7&

\scalebox{.7}{\begin{tikzpicture}
\node (1) at (-2,0)[vertex]{};
\node (2) at (-1,0)[vertex]{};
\node (3) at (0,0)[vertex]{};
\node (4) at (0,1)[vertex]{};
\node (5) at (1,0)[vertex]{};
\node (6) at (2,0)[vertex]{};
\node (7) at (0,1.3){};
\node (8) at (3,0){};
\draw[->](1)--(2);
\draw[->](2)--(3);
\draw[->](6)--(5);
\draw[->](5)--(3);
\draw[-](3)--(4);
\end{tikzpicture}\label{(2,2,1,1)_1}}
\scalebox{.7}{\begin{tikzpicture}
\node (1) at (-2,0)[vertex]{};
\node (2) at (-1,0)[vertex]{};
\node (3) at (0,0)[vertex]{};
\node (4) at (0,1)[vertex]{};
\node (5) at (1,0)[vertex]{};
\node (6) at (2,0)[vertex]{};
\node (7) at (0,1.3){};
\draw[->](1)--(2);
\draw[->](3)--(2);
\draw[->](6)--(5);
\draw[->](3)--(5);
\draw[-](3)--(4);
\end{tikzpicture}\label{(2,2,1,1)_2}}

\scalebox{.7}{\begin{tikzpicture}
\node (1) at (-2,0)[vertex]{};
\node (2) at (-1,0)[vertex]{};
\node (3) at (0,0)[vertex]{};
\node (4) at (0,1)[vertex]{};
\node (5) at (1,0)[vertex]{};
\node (6) at (2,0)[vertex]{};
\node (7) at (0,1.3){};
\node (8) at (3,0){};
\draw[->](2)--(1);
\draw[->](2)--(3);
\draw[->](5)--(6);
\draw[->](5)--(3);
\draw[-](3)--(4);
\end{tikzpicture}\label{(2,2,1,1)_3}}
\scalebox{.7}{\begin{tikzpicture}
\node (1) at (-2,0)[vertex]{};
\node (2) at (-1,0)[vertex]{};
\node (3) at (0,0)[vertex]{};
\node (4) at (0,1)[vertex]{};
\node (5) at (1,0)[vertex]{};
\node (6) at (2,0)[vertex]{};
\node (7) at (0,1.3){};
\draw[->](2)--(1);
\draw[->](3)--(2);
\draw[->](5)--(6);
\draw[->](3)--(5);
\draw[-](3)--(4);
\end{tikzpicture}\label{(2,2,1,1)_4}}
\\ \hline
(2+2,0,1,1)&s=7&
\scalebox{.7}{\begin{tikzpicture}
\node (a) at (0,0)[vertex]{};
\node (b) at (2,0)[vertex]{};
\node (c) at (1,1)[vertex]{};
\node (d) at (4,0)[vertex]{};
\node (e) at (3,1)[vertex]{};
\node (f) at (2,1)[vertex]{};
\node (0) at (0,1.3){};
\draw[-](b)--(f);
\draw[->](a)-- (b);
\draw[->](b)--(c);
\draw[->](c)--(a);
\draw[->](d)--(b);
\draw[->](b)--(e);
\draw[->](e)--(d);
\end{tikzpicture}\label{(2+2,0,1,1)_1}}
\\ \hline
(2,2,2,0)&s=7&
\scalebox{.7}{\begin{tikzpicture}
\node(a) at (0,0)[vertex]{};
\node(b) at (1,0)[vertex]{};
\node(c) at (2,1)[vertex]{};
\node(d) at (3,0)[vertex]{};
\node(e) at (2,-1)[vertex]{};
\node(f) at (4,0)[vertex]{};
\node (0) at (0,1.3){};
\draw[-](a)--(b);
\draw[->](b)--(c);
\draw[->](b)--(d);
\draw[->](b)--(e);
\draw[->](e)--(d);
\draw[-](d)--(f);
\draw[->](c)--(d);
\end{tikzpicture}\label{(2+2,0,1,1)}}
\\ \hline
(2+2,0,2,0) & s=7&
\scalebox{.7}{\begin{tikzpicture}
\node(a) at (0,0)[vertex]{};
\node(b) at (0,-1.5)[vertex]{};
\node(c) at (1.5,-1.5)[vertex]{};
\node(d) at (1.5,0)[vertex]{};
\node(e) at (3,-1.5)[vertex]{};
\node(f) at (3,0)[vertex]{};
\draw[->](a)--(d);
\draw[->](f)--(d);
\draw[->](b)--(a);
\draw[->](e)--(f);
\draw[->](c)--(b);
\draw[->](c)--(e);
\draw[->](c)--(d);
\draw[->](d)--(b);
\draw[->](d)--(e);
\end{tikzpicture}\label{(2+2,0,2,0)_1}}
\scalebox{.7}{\begin{tikzpicture}
\node(0) at (0,0.3){};
\node(a) at (0,0)[vertex]{};
\node(b) at (0,-1.5)[vertex]{};
\node(c) at (1.5,-1.5)[vertex]{};
\node(d) at (1.5,0)[vertex]{};
\node(e) at (3,-1.5)[vertex]{};
\node(f) at (3,0)[vertex]{};
\node(00) at (3.3,0){};
\node(000) at (-0.3,0){};
\draw[->](d)--(a);
\draw[->](d)--(f);
\draw[->](a)--(b);
\draw[->](f)--(e);
\draw[->](b)--(c);
\draw[->](e)--(c);
\draw[->](d)--(c);
\draw[->](b)--(d);
\draw[->](e)--(d);
\end{tikzpicture}\label{(2+2,0,2,0)_2}}
\scalebox{.7}{\begin{tikzpicture}
\node(a) at (0,0)[vertex]{};
\node(b) at (0,-1.5)[vertex]{};
\node(c) at (1.5,-1.5)[vertex]{};
\node(d) at (1.5,0)[vertex]{};
\node(e) at (3,-1.5)[vertex]{};
\node(f) at (3,0)[vertex]{};
\draw[->](d)--(a);
\draw[->](d)--(f);
\draw[->](a)--(b);
\draw[->](f)--(e);
\draw[->](b)--(c);
\draw[->](e)--(c);
\draw[->](d)--(c);
\end{tikzpicture}\label{(2+2,0,2,0)_3}}
\\ \hline
\end{tabular}
\centering\caption{All quivers of \(2\)-cluster tilting objects of \(\mathcal C_2(E_6)\) where the corresponding \(2\)-cluster tilting object is periodic under \(\tau^7\). The first column gives the distribution of the indecomposable summands of \(T\) in the \(\tau\)-orbits of the AR-quiver of \(\mathcal C_2(E_6)\). }
\end{table}\label{tableE6}


\subsection{Type \(E_7\)}
The AR-quiver of \(\mathcal C_2(E_7)\) is illustrated in figure \ref{ARQuiverE7}, it has the shape of a cylinder where each \(\tau\)-orbit contains 10 objects. In order to achieve any relation of the type \(\tau^s \mathcal T=\mathcal T\)  for a \(2\)-cluster tilting subcategory \(\mathcal T\) of \(\mathcal C_2(E_7)\), we see that the only possibilities are \(s=2\) or \(s=5\). 

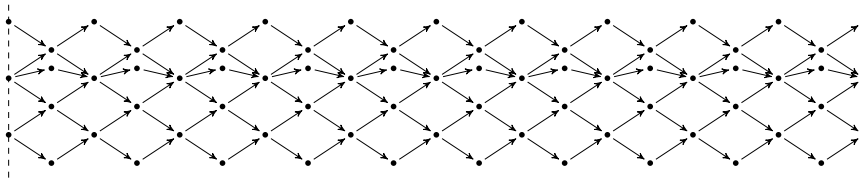
\begin{figure}[h]
 \scalebox{.75}{ \begin{tikzpicture}[baseline=-1.7cm,scale=.5,yscale=-1]
 \foreach \x in {0,...,10}
  \foreach \y in {1,3,5}
   \node (\y-\x) at (\x*3,\y) [vertex] {};
 \foreach \x in {0,...,9}
  \foreach \y in {2,4,6}
   \node (\y-\x) at (\x*3+1.5,\y) [vertex] {};
 \foreach \x in {0,...,9}
  \foreach \y in {3}
   \node (a-\x) at (\x*3+1.5,\y-0.35) [vertex] {};
 \foreach \xa/\xb in {0/1,1/2,2/3,3/4,4/5,5/6,6/7,7/8,8/9,9/10}
  \foreach \ya/\yb in {1/2,3/2,3/4,5/4,5/6}
   {
    \draw [->] (\ya-\xa) -- (\yb-\xa);
    \draw [->] (\yb-\xa) -- (\ya-\xb);
   }
   \foreach \xa/\xb in {0/1,1/2,2/3,3/4,4/5,5/6,6/7,7/8,8/9,9/10}
    \foreach \ya/\yb in {3/3}
    {
     \draw [->] (\ya-\xa)--(a-\xa);
     \draw [->] (a-\xa)--(\yb-\xb);   
    }
 \draw [dashed] (0,0.4) -- (0,6.6); 
 \draw [dashed] (30,0.4) -- (30,6.6);
\end{tikzpicture} }\caption{The AR-quiver of \(\mathcal C_2(E_7)\).}\label{ARQuiverE7} 
\end{figure}

Assume that \(s=2\). In that case each \(\tau\)-orbit containing at least one indecomposable summand of \(\mathcal T\) contains a number of indecomposables from \(\mathcal T\) divisible by 5. However any \(2\)-cluster tilting subcategory of \(\mathcal C_2(E_7)\) has exactly \(7\) indecomposable objects up to isomorphism, hence this is an impossible scenario.

Now assume that \(s=5\). Then all \(\tau\)-orbits containing at least one indecomposable summand of \(\mathcal T\), contains an even number of summands of \(\mathcal T\). This can not happen as \(7\) is odd. Hence we have the following result:

\begin{theorem}
If \(\mathcal T\) is any \(2\)-cluster tilting subcategory of \(\mathcal C_2(E_7)\), then the smallest value of \(s\) such that \(\tau^s\mathcal T=\mathcal T\) is \(s=10\).
\end{theorem}

\subsection{Type \(E_8\)}\label{subsecE8} The AR-quiver of \(\mathcal C_2(E_8)\) has \(8\) \(\tau\)-orbits, each with \(16\) indecomposable objects, as illustrated in figure \ref{ARQuiverE8}. If \(T\) is a \(2\)-cluster tilting object of \(\mathcal C_2(E_8)\) such that \(\tau^sT=T\) with \(s<16\), then \(s\in\{2,4,8\}\). Below we have included for each \(\tau\)-orbit, a figure where the grey area shows which indecomposables that can not be part of the same \(2\)-cluster tilting object as \(M\). By counting in each figure below, one may verify that it is not possible to have \(\tau^2T=T\) for a \(2\)-cluster tilting object in \(\mathcal C_2(E_8)\).

Furthermore, if \(T\) is a \(2\)-cluster tilting object such that \(\tau^4T=T\), then the only \(\tau\)-orbits that can contain indecomposable summands of \(T\) are the two bottom \(\tau\)-orbits in figure \ref{ARQuiverE8}. 

Finally we study the case when \(s=8\). Again we turn to the 8 figures displaying the ext-support of an object of each \(\tau\)-orbit. From these figures we see that if \(\tau^8T=T\) for a \(2\)-cluster tilting object, then the indecomposable objects of \(T\) may only occur in the two bottom \(\tau\)-orbits, or in the top \(\tau\)-orbit.

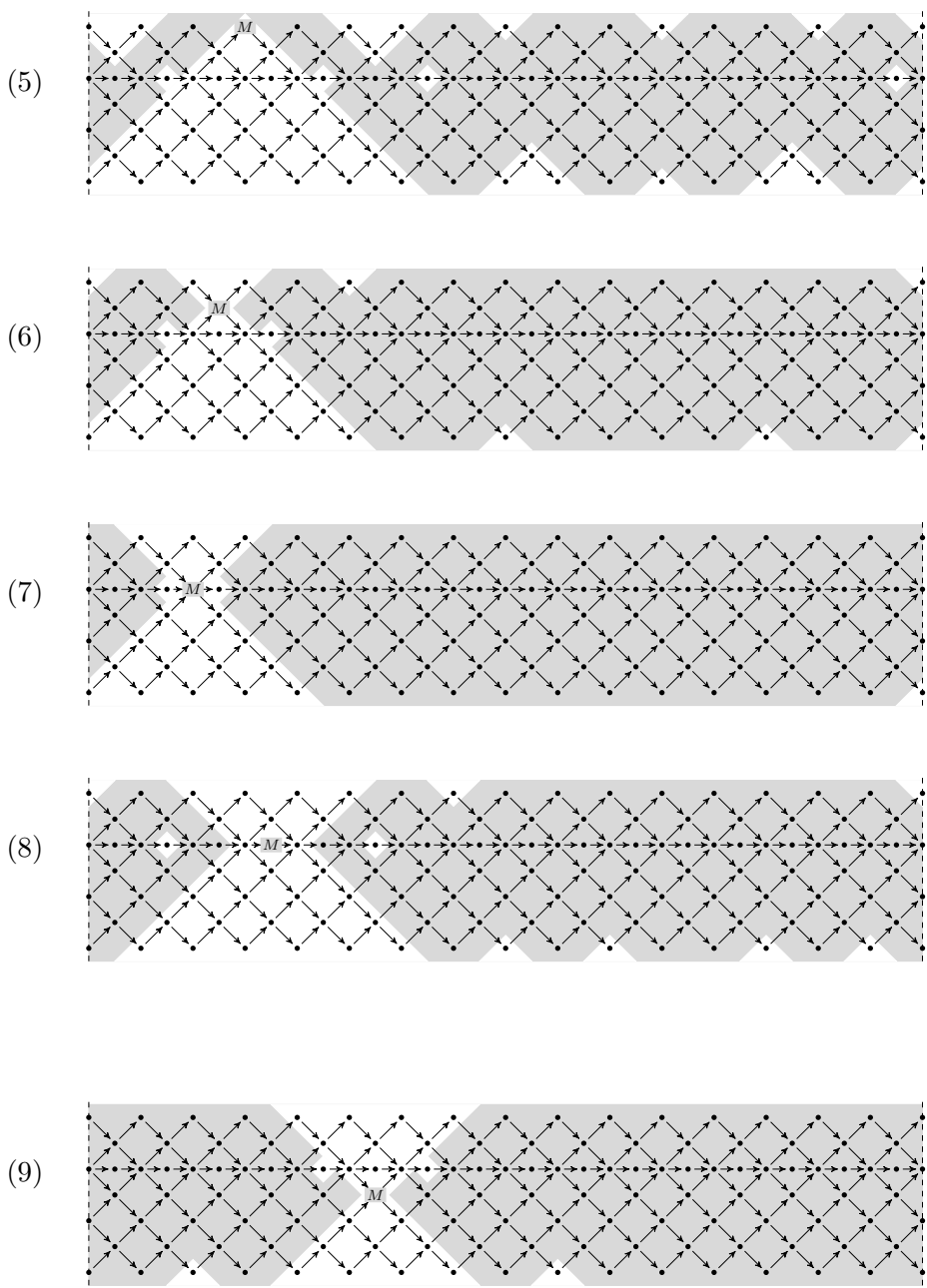
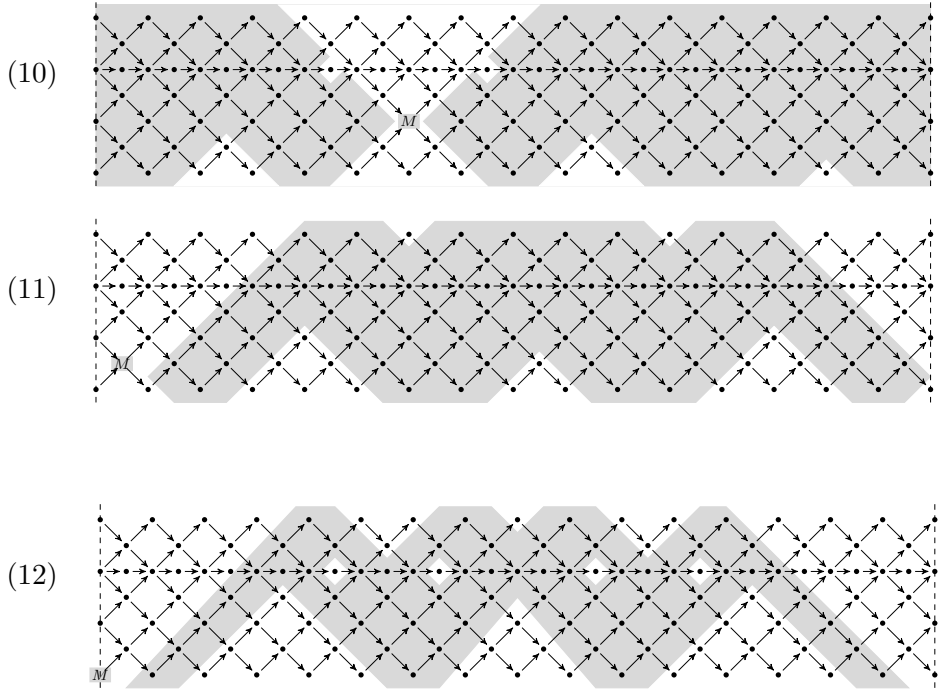
\begin{figure}
\scalebox{.75}{ \begin{tikzpicture}[scale=.5,yscale=-1]
 \foreach \x in {0,...,16}
  \foreach \y in {1,3,5,7}
   \node (\y-\x) at (\x*2,\y) [vertex] {};
 \foreach \x in {0,...,15}
  \foreach \y in {2,4,6}
   \node (\y-\x) at (\x*2+1,\y) [vertex] {};
 \foreach \x in {0,...,15}
  \foreach \y in {3}
   \node (a-\x) at (\x*2+1,\y) [vertex] {}; 
 \foreach \xa/\xb in {0/1,1/2,2/3,3/4,4/5,5/6,6/7,7/8,8/9,9/10,10/11,11/12,12/13,13/14,14/15,15/16}
  \foreach \ya/\yb in {1/2,3/2,3/4,5/4,5/6,7/6}
   {
    \draw [->] (\ya-\xa) -- (\yb-\xa);
    \draw [->] (\yb-\xa) -- (\ya-\xb);
   }
 \foreach \xa/\xb in {0/1,1/2,2/3,3/4,4/5,5/6,6/7,7/8,8/9,9/10,10/11,11/12,12/13,13/14,14/15,15/16}
    \foreach \ya/\yb in {3/3}
   {
     \draw [->] (\ya-\xa)--(a-\xa);
     \draw [->] (a-\xa) --(\ya-\xb);
   }     
 \draw [dashed] (0,0.4) -- (0,7.6); 
 \draw [dashed] (32,0.4) -- (32,7.6);
\end{tikzpicture} }\caption{The AR-quiver of \(\mathcal C_2(E_8)\).}\label{ARQuiverE8}
\end{figure}
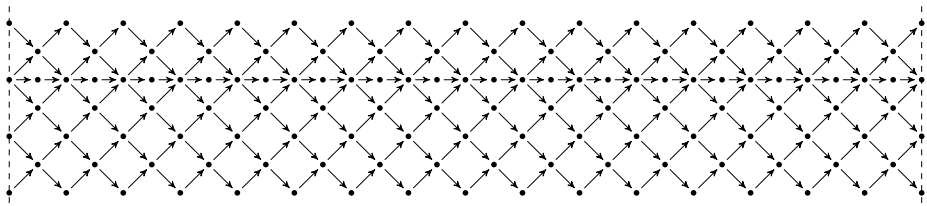


\begin{equation}
 \scalebox{.7}{ \begin{tikzpicture}[baseline=-1.7cm,scale=.49,yscale=-1]\label{E8_P1}
    \fill[fill1] (0,0.5)--(32,0.5)--(32,7.5)--(0,7.5)--cycle;
 \fill[fill0] (0,7.5)--(0,6.5)--(6,0.5)--(13,7.5)--cycle;
 \fill[fill0] (9,0.5)--(11,2.5)--(13,0.5)--cycle;
 \fill[fill0] (0,0.5)--(0,1.5)--(1,2.5)--(3,0.5)--cycle;
 \fill[fill0] (31,0.5)--(32,1.5)--(32,0.5)--cycle;
 \fill[fill0] (31,7.5)--(32,6.5)--(32,7.5)--cycle; 
 \fill[fill0] (2.5,3)--(3,2.5)--(3.5,3)--(3,3.5)--cycle;
 \fill[fill0] (12.5,3)--(13,2.5)--(13.5,3)--(13,3.5)--cycle; 
 \fill[fill0] (30.5,3)--(31,2.5)--(31.5,3)--(31,3.5)--cycle;
 \fill[fill0] (8.5,3)--(9,2.5)--(9.5,3)--(9,3.5)--cycle; 
 \fill[fill0] (15,7.5)--(17,5.5)--(19,7.5)--cycle;
 \fill[fill0] (25,7.5)--(27,5.5)--(29,7.5)--cycle;
 \fill[fill0] (15,0.5)--(16,1.5)--(17,0.5)--cycle;
 \fill[fill0](21,0.5)--(22,1.5)--(23,0.5)--cycle;
 \fill[fill0] (27,0.5)--(28,01.5)--(29,0.5)--cycle; 
 \fill[fill0] (21,7.5)--(22,6.5)--(23,7.5)--cycle; 
 \foreach \x in {0,...,16}
  \foreach \y in {1,3,5,7}
   \node (\y-\x) at (\x*2,\y) [vertex] {};
 \foreach \x in {0,...,15}
  \foreach \y in {2,4,6}
   \node (\y-\x) at (\x*2+1,\y) [vertex] {};
 \foreach \x in {0,...,15}
  \foreach \y in {3}
   \node (a-\x) at (\x*2+1,\y) [vertex] {}; 
   \replacevertex{(1-3)}{[tvertex][fill1] {$M$}} 
 \foreach \xa/\xb in {0/1,1/2,2/3,3/4,4/5,5/6,6/7,7/8,8/9,9/10,10/11,11/12,12/13,13/14,14/15,15/16}
  \foreach \ya/\yb in {1/2,3/2,3/4,5/4,5/6,7/6}
   {
    \draw [->] (\ya-\xa) -- (\yb-\xa);
    \draw [->] (\yb-\xa) -- (\ya-\xb);
   }
 \foreach \xa/\xb in {0/1,1/2,2/3,3/4,4/5,5/6,6/7,7/8,8/9,9/10,10/11,11/12,12/13,13/14,14/15,15/16}
    \foreach \ya/\yb in {3/3}
   {
     \draw [->] (\ya-\xa)--(a-\xa);
     \draw [->] (a-\xa) --(\ya-\xb);
   }      
 \draw [dashed] (0,0.4) -- (0,7.6); 
 \draw [dashed] (32,0.4) -- (32,7.6);
\end{tikzpicture} } 
\end{equation}

\begin{equation}
 \scalebox{.7}{ \begin{tikzpicture}[baseline=-1.7cm,scale=.49,yscale=-1]\label{E8_P2}
    \fill[fill1] (0,0.5)--(32,0.5)--(32,7.5)--(0,7.5)--cycle;
 \fill[fill0] (3,0.5)--(5,2.5)--(7,0.5)--cycle;
 \fill[fill0] (0,7.5)--(0,6.5)--(5,1.5)--(11,7.5)--cycle;
 \fill[fill0] (0,0.5)--(1,0.5)--(0,1.5)--cycle;
 \fill[fill0] (31,0.5)--(32,1.5)--(32,0.5)--cycle;
  \fill[fill0] (31,7.5)--(32,6.5)--(32,7.5)--cycle; 
 \fill[fill0] (2.5,3)--(3,2.5)--(3.5,3)--(3,3.5)--cycle;
 \fill[fill0] (6.5,3)--(7,2.5)--(7.5,3)--(7,3.5)--cycle;
 \fill[fill0] (15,7.5)--(16,6.5)--(17,7.5)--(16,7.5)--cycle;
 \fill[fill0] (25,7.5)--(26,6.5)--(26.5,7)--(27,7.5)--cycle;
 \fill[fill0] (9,0.5)--(10,1.5)--(11,0.5)--cycle; 
 \foreach \x in {0,...,16}
  \foreach \y in {1,3,5,7}
   \node (\y-\x) at (\x*2,\y) [vertex] {};
 \foreach \x in {0,...,15}
  \foreach \y in {2,4,6}
   \node (\y-\x) at (\x*2+1,\y) [vertex] {};
 \foreach \x in {0,...,15}
  \foreach \y in {3}
   \node (a-\x) at (\x*2+1,\y) [vertex] {};
   
   \replacevertex{(2-2)}{[tvertex][fill1] {$M$}} 
 \foreach \xa/\xb in {0/1,1/2,2/3,3/4,4/5,5/6,6/7,7/8,8/9,9/10,10/11,11/12,12/13,13/14,14/15,15/16}
  \foreach \ya/\yb in {1/2,3/2,3/4,5/4,5/6,7/6}
   {
    \draw [->] (\ya-\xa) -- (\yb-\xa);
    \draw [->] (\yb-\xa) -- (\ya-\xb);
   }
 \foreach \xa/\xb in {0/1,1/2,2/3,3/4,4/5,5/6,6/7,7/8,8/9,9/10,10/11,11/12,12/13,13/14,14/15,15/16}
    \foreach \ya/\yb in {3/3}
   {
     \draw [->] (\ya-\xa)--(a-\xa);
     \draw [->] (a-\xa) --(\ya-\xb);
   }     
 \draw [dashed] (0,0.4) -- (0,7.6); 
 \draw [dashed] (32,0.4) -- (32,7.6);
\end{tikzpicture} } 
\end{equation}
%

\begin{equation}
 \scalebox{.7}{ \begin{tikzpicture}[baseline=-1.7cm,scale=.49,yscale=-1]\label{E8_P3}
    \fill[fill1] (0,0.5)--(32,0.5)--(32,7.5)--(0,7.5)--cycle;
 \fill[fill0] (1,0.5)--(4,3.5)--(7,0.5)--cycle; 
 \fill[fill0] (0,7.5)--(0,6.5)--(4,2.5)--(9,7.5)--cycle; 
 \fill[fill0] (31,7.5)--(32,6.5)--(32,7.5)--cycle; 
 \fill[fill0] (2.5,3)--(3,2.5)--(3.5,3)--(3,3.5)--cycle;
 \fill[fill0] (4.5,3)--(5,2.5)--(5.5,3)--(5,3.5)--cycle; 
 \foreach \x in {0,...,16}
  \foreach \y in {1,3,5,7}
   \node (\y-\x) at (\x*2,\y) [vertex] {};
 \foreach \x in {0,...,15}
  \foreach \y in {2,4,6}
   \node (\y-\x) at (\x*2+1,\y) [vertex] {};
 \foreach \x in {0,...,15}
  \foreach \y in {3}
   \node (a-\x) at (\x*2+1,\y) [vertex] {};
   
   \replacevertex{(3-2)}{[tvertex][fill1] {$M$}} 
 \foreach \xa/\xb in {0/1,1/2,2/3,3/4,4/5,5/6,6/7,7/8,8/9,9/10,10/11,11/12,12/13,13/14,14/15,15/16}
  \foreach \ya/\yb in {1/2,3/2,3/4,5/4,5/6,7/6}
   {
    \draw [->] (\ya-\xa) -- (\yb-\xa);
    \draw [->] (\yb-\xa) -- (\ya-\xb);
   }
 \foreach \xa/\xb in {0/1,1/2,2/3,3/4,4/5,5/6,6/7,7/8,8/9,9/10,10/11,11/12,12/13,13/14,14/15,15/16}
    \foreach \ya/\yb in {3/3}
   {
     \draw [->] (\ya-\xa)--(a-\xa);
     \draw [->] (a-\xa) --(\ya-\xb);
   }    
 \draw [dashed] (0,0.4) -- (0,7.6); 
 \draw [dashed] (32,0.4) -- (32,7.6);
\end{tikzpicture} } 
\end{equation}
%

\begin{equation}
 \scalebox{.7}{ \begin{tikzpicture}[baseline=-1.7cm,scale=.49,yscale=-1]
    \fill[fill1] (0,0.5)--(32,0.5)--(32,7.5)--(0,7.5)--cycle;
 \fill[fill0] (3,0.5)--(7,4.5)--(11,0.5)--cycle;
 \fill[fill0] (1,7.5)--(7,1.5)--(13,7.5)--cycle;
 \fill[fill0] (0,0.5)--(1,0.5)--(0,1.5)--cycle;
 \fill[fill0] (31,0.5)--(32,1.5)--(32,0.5)--cycle;
 \fill[fill0] (2.5,3)--(3,2.5)--(3.5,3)--(3,3.5)--cycle;
 \fill[fill0] (10.5,3)--(11,2.5)--(11.5,3)--(11,3.5)--cycle;
 \fill[fill0] (15,7.5)--(16,6.5)--(17,7.5)--cycle;
 \fill[fill0] (19,7.5)--(20,6.5)--(21,7.5)--cycle;
 \fill[fill0] (25,7.5)--(26,6.5)--(27,7.5)--cycle;
 \fill[fill0] (29,7.5)--(30,6.5)--(31,7.5)--cycle;
 \fill[fill0] (13,0.5)--(14,1.5)--(15,0.5)--cycle;
 \foreach \x in {0,...,16}
  \foreach \y in {1,3,5,7}
   \node (\y-\x) at (\x*2,\y) [vertex] {};
 \foreach \x in {0,...,15}
  \foreach \y in {2,4,6}
   \node (\y-\x) at (\x*2+1,\y) [vertex] {};
 \foreach \x in {0,...,15}
  \foreach \y in {3}
   \node (a-\x) at (\x*2+1,\y) [vertex] {};
   
   \replacevertex{(a-3)}{[tvertex][fill1] {$M$}} 
 \foreach \xa/\xb in {0/1,1/2,2/3,3/4,4/5,5/6,6/7,7/8,8/9,9/10,10/11,11/12,12/13,13/14,14/15,15/16}
  \foreach \ya/\yb in {1/2,3/2,3/4,5/4,5/6,7/6}
   {
    \draw [->] (\ya-\xa) -- (\yb-\xa);
    \draw [->] (\yb-\xa) -- (\ya-\xb);
   }
 \foreach \xa/\xb in {0/1,1/2,2/3,3/4,4/5,5/6,6/7,7/8,8/9,9/10,10/11,11/12,12/13,13/14,14/15,15/16}
    \foreach \ya/\yb in {3/3}
   {
     \draw [->] (\ya-\xa)--(a-\xa);
     \draw [->] (a-\xa) --(\ya-\xb);
   }      
 \draw [dashed] (0,0.4) -- (0,7.6); 
 \draw [dashed] (32,0.4) -- (32,7.6);
\end{tikzpicture} } 
\end{equation}\\

\begin{equation}
 \scalebox{.7}{ \begin{tikzpicture}[baseline=-1.7cm,scale=.49,yscale=-1]\label{E8_P4}
    \fill[fill1] (0,0.5)--(32,0.5)--(32,7.5)--(0,7.5)--cycle;
 \fill[fill0] (7,0.5)--(11,4.5)--(15,0.5)--cycle; 
 \fill[fill0] (7,7.5)--(11,3.5)--(15,7.5)--cycle; 
 \fill[fill0] (3,7.5)--(4,6.5)--(4.5,7)--(5,7.5)--cycle; 
 \fill[fill0] (8.5,3)--(9,2.5)--(9.5,3)--(9,3.5)--cycle;
 \fill[fill0] (12.5,3)--(13,2.5)--(13.5,3)--(13,3.5)--cycle;
 \fill[fill0] (17,7.5)--(18,6.5)--(18.5,7)--(19,7.5)--cycle;
 \foreach \x in {0,...,16}
  \foreach \y in {1,3,5,7}
   \node (\y-\x) at (\x*2,\y) [vertex] {};
 \foreach \x in {0,...,15}
  \foreach \y in {2,4,6}
   \node (\y-\x) at (\x*2+1,\y) [vertex] {};
 \foreach \x in {0,...,15}
  \foreach \y in {3}
   \node (a-\x) at (\x*2+1,\y) [vertex] {}; 
   
   \replacevertex{(4-5)}{[tvertex][fill1] {$M$}}
 \foreach \xa/\xb in {0/1,1/2,2/3,3/4,4/5,5/6,6/7,7/8,8/9,9/10,10/11,11/12,12/13,13/14,14/15,15/16}
  \foreach \ya/\yb in {1/2,3/2,3/4,5/4,5/6,7/6}
   {
    \draw [->] (\ya-\xa) -- (\yb-\xa);
    \draw [->] (\yb-\xa) -- (\ya-\xb);
   }
 \foreach \xa/\xb in {0/1,1/2,2/3,3/4,4/5,5/6,6/7,7/8,8/9,9/10,10/11,11/12,12/13,13/14,14/15,15/16}
    \foreach \ya/\yb in {3/3}
   {
     \draw [->] (\ya-\xa)--(a-\xa);
     \draw [->] (a-\xa) --(\ya-\xb);
   }     
 \draw [dashed] (0,0.4) -- (0,7.6); 
 \draw [dashed] (32,0.4) -- (32,7.6);
\end{tikzpicture} } 
\end{equation}

\begin{equation}
 \scalebox{.7}{ \begin{tikzpicture}[baseline=-1.7cm,scale=.49,yscale=-1]\label{E8_P5}
   \fill[fill1] (0,0.5)--(32,0.5)--(32,7.5)--(0,7.5)--cycle;
 \fill[fill0] (3,7.5)--(5,5.5)--(7,7.5)--cycle;
 \fill[fill0] (9,7.5)--(12,4.5)--(15,7.5)--cycle;
 \fill[fill0] (7,0.5)--(12,5.5)--(17,0.5)--cycle;
 \fill[fill0] (17,7.5)--(19,5.5)--(21,7.5)--cycle;
  \fill[fill0] (27.5,7)--(28,6.5)--(29,7.5)--(27,7.5)--cycle;
  \fill[fill0] (8.5,3)--(9,2.5)--(9.5,3)--(9,3.5)--cycle; 
  \fill[fill0] (14.5,3)--(15,2.5)--(15.5,3)--(15,3.5)--cycle;
 \foreach \x in {0,...,16}
  \foreach \y in {1,3,5,7}
   \node (\y-\x) at (\x*2,\y) [vertex] {};
 \foreach \x in {0,...,15}
  \foreach \y in {2,4,6}
   \node (\y-\x) at (\x*2+1,\y) [vertex] {};
 \foreach \x in {0,...,15}
  \foreach \y in {3}
   \node (a-\x) at (\x*2+1,\y) [vertex] {};
   
   \replacevertex{(5-6)}{[tvertex][fill1] {$M$}}
 \foreach \xa/\xb in {0/1,1/2,2/3,3/4,4/5,5/6,6/7,7/8,8/9,9/10,10/11,11/12,12/13,13/14,14/15,15/16}
  \foreach \ya/\yb in {1/2,3/2,3/4,5/4,5/6,7/6}
   {
    \draw [->] (\ya-\xa) -- (\yb-\xa);
    \draw [->] (\yb-\xa) -- (\ya-\xb);
   }
 \foreach \xa/\xb in {0/1,1/2,2/3,3/4,4/5,5/6,6/7,7/8,8/9,9/10,10/11,11/12,12/13,13/14,14/15,15/16}
    \foreach \ya/\yb in {3/3}
   {
     \draw [->] (\ya-\xa)--(a-\xa);
     \draw [->] (a-\xa) --(\ya-\xb);
   }     
 \draw [dashed] (0,0.4) -- (0,7.6); 
 \draw [dashed] (32,0.4) -- (32,7.6);
\end{tikzpicture} } 
\end{equation}

\begin{equation}
 \scalebox{.7}{ \begin{tikzpicture}[baseline=-1.7cm,scale=.49,yscale=-1]
\fill[fill1](2,6.5)--(3,7.5)--(5,7.5)--(8,4.5)--(11,7.5)--(15,7.5)--(17,5.5)--(19,7.5)--(23,7.5)--(26,4.5)--(29,7.5)--(31,7.5)--(32,6.5)--(26,0.5)--(23,0.5)--(22,1.5)--(21,0.5)--(13,0.5)--(12,1.5)--(11,0.5)--(8,0.5)--cycle;

 \foreach \x in {0,...,16}
  \foreach \y in {1,3,5,7}
   \node (\y-\x) at (\x*2,\y) [vertex] {};
 \foreach \x in {0,...,15}
  \foreach \y in {2,4,6}
   \node (\y-\x) at (\x*2+1,\y) [vertex] {};
 \foreach \x in {0,...,15}
  \foreach \y in {3}
   \node (a-\x) at (\x*2+1,\y) [vertex] {};
   
 \replacevertex{(1,6)}{[tvertex][fill1] {$M$}} 
 \foreach \xa/\xb in {0/1,1/2,2/3,3/4,4/5,5/6,6/7,7/8,8/9,9/10,10/11,11/12,12/13,13/14,14/15,15/16}
  \foreach \ya/\yb in {1/2,3/2,3/4,5/4,5/6,7/6}
   {
    \draw [->] (\ya-\xa) -- (\yb-\xa);
    \draw [->] (\yb-\xa) -- (\ya-\xb);
   }
 \foreach \xa/\xb in {0/1,1/2,2/3,3/4,4/5,5/6,6/7,7/8,8/9,9/10,10/11,11/12,12/13,13/14,14/15,15/16}
    \foreach \ya/\yb in {3/3}
   {
     \draw [->] (\ya-\xa)--(a-\xa);
     \draw [->] (a-\xa) --(\ya-\xb);
   }      
 \draw [dashed] (0,0.4) -- (0,7.6); 
 \draw [dashed] (32,0.4) -- (32,7.6);
\end{tikzpicture} } 
\end{equation}\\

\begin{equation}
 \scalebox{.7}{ \begin{tikzpicture}[baseline=-1.7cm,scale=.49,yscale=-1]
 \fill[fill1] (1,7.5)--(3,7.5)--(7,3.5)--(11,7.5)--(13,7.5)--(16,4)--(19,7.5)--(21,7.5)--(25,3.5)--(29,7.5)--(31,7.5)--(24,0.5)--(23,0.5)--(21,2.5)--(19,0.5)--(17,0.5)--(16,1.5)--(15,0.5)--(13,0.5)--(11,2.5)--(9,0.5)--(7.5,0.5)--cycle;
 \fill[fill0](8.5,3)--(9,2.5)--(9.5,3)--(9,3.5)--cycle;
 \fill[fill0](12.5,3)--(13,2.5)--(13.5,3)--(13,3.5)--cycle;
 \fill[fill0](18.5,3)--(19,2.5)--(19.5,3)--(19,3.5)--cycle;
 \fill[fill0](22.5,3)--(23,2.5)--(23.5,3)--(23,3.5)--cycle;
 \foreach \x in {0,...,16}
  \foreach \y in {1,3,5,7}
   \node (\y-\x) at (\x*2,\y) [vertex] {};
 \foreach \x in {0,...,15}
  \foreach \y in {2,4,6}
   \node (\y-\x) at (\x*2+1,\y) [vertex] {};
 \foreach \x in {0,...,15}
  \foreach \y in {3}
   \node (a-\x) at (\x*2+1,\y) [vertex] {};
   
   \replacevertex{(0,7)}{[tvertex][fill1] {$M$}} 
 \foreach \xa/\xb in {0/1,1/2,2/3,3/4,4/5,5/6,6/7,7/8,8/9,9/10,10/11,11/12,12/13,13/14,14/15,15/16}
  \foreach \ya/\yb in {1/2,3/2,3/4,5/4,5/6,7/6}
   {
    \draw [->] (\ya-\xa) -- (\yb-\xa);
    \draw [->] (\yb-\xa) -- (\ya-\xb);
   }
 \foreach \xa/\xb in {0/1,1/2,2/3,3/4,4/5,5/6,6/7,7/8,8/9,9/10,10/11,11/12,12/13,13/14,14/15,15/16}
    \foreach \ya/\yb in {3/3}
   {
     \draw [->] (\ya-\xa)--(a-\xa);
     \draw [->] (a-\xa) --(\ya-\xb);
   }      
 \draw [dashed] (0,0.4) -- (0,7.6); 
 \draw [dashed] (32,0.4) -- (32,7.6);
\end{tikzpicture} } 
\end{equation}\\


The number of possible distributions of the \(8\) indecomposable summands are now significantly reduced. The remaining possible distributions have all been checked, and the result is presented in table \ref{tableE8}. Before proceeding to the table we present the necessary notation, as well as an example of the calculations. 

For the rest of the subsection we will call the bottom \(\tau\)-orbit of the AR-quiver nr 1, the second bottom \(\tau\)-orbit nr 2 and the top \(\tau\)-orbit nr 8. We will now introduce a short notation for the distribution of summands between these three \(\tau\)-orbits. Denote the distribution of indecomposables by the vector \((a,b,c)\), where \(a\) is the indecomposables in \(\tau\)-orbit nr 1, \(b\) the indecomposables in \(\tau\)-orbit nr 2 and \(c\) the indecomposables in \(\tau\)-orbit nr 8.  If a \(2\)-cluster tilting object \(T\) has 2 indecomposables closed under \(\tau^8\) in a \(\tau\)-orbit, it will be denoted 2. If \(T\) has 4 indecomposables such that these are closed under \(\tau^4\) it will be denoted 4. If there are 4 indecomposables in a \(\tau\)-orbit but these indecomposables are not closed under \(\tau^4\) but closed under \(\tau^8\) we will denote these objects by 2+2. 

We now demonstrate an example of the calculations that have been done to achieve table \ref{tableE8}, by considering the cases \((?,2+2,?)\) and \((?,4,?)\). First we consider the situation if there are two indecomposables of \(T\) placed in \(\tau\)-orbit nr 2. This is pictured in figure \ref{E8_eks_del1}.

\begin{figure}[h]
\scalebox{.7}{ \begin{tikzpicture}[baseline=-1.7cm,scale=.49,yscale=-1]
\fill[fill3](2,6.5)--(3,7.5)--(5,7.5)--(8,4.5)
--(11,7.5)--(15,7.5)--(17,5.5)--(19,7.5)--(23,7.5)--(26,4.5)--(29,7.5)--(31,7.5)--(32,6.5)--(26,0.5)--
(23,0.5)--(22,1.5)--(21,0.5)--(13,0.5)--(12,1.5)--(11,0.5)
--(8,0.5)--cycle;

\fill[fill2](18,6.5)--(19,7.5)--(21,7.5)--(24,4.5)--(27,7.5)--(31,7.5)--(32,6.5)--(32,0.5)--(29,0.5)--(28,1.5)--(27,0.5)--(24,0.5)--cycle;
\fill[fill2](0,6.5)--(1,5.5)--(3,7.5)--(7,7.5)--(10,4.5)--(13,7.5)--(15,7.5)--(16,6.5)--(10,0.5)--(7,0.5)--(6,1.5)--(5,0.5)--(0,0.5)--cycle;
 \foreach \x in {0,...,16}
  \foreach \y in {1,3,5,7}
   \node (\y-\x) at (\x*2,\y) [vertex] {};
 \foreach \x in {0,...,15}
  \foreach \y in {2,4,6}
   \node (\y-\x) at (\x*2+1,\y) [vertex] {};
 \foreach \x in {0,...,15}
  \foreach \y in {3}
   \node (a-\x) at (\x*2+1,\y) [vertex] {};
   
 \replacevertex{(1,6)}{[tvertex][fill3] {$M$}} 
 \replacevertex{(17,6)}{[tvertex][fill2]{$N$}}
 \foreach \xa/\xb in {0/1,1/2,2/3,3/4,4/5,5/6,6/7,7/8,8/9,9/10,10/11,11/12,12/13,13/14,14/15,15/16}
  \foreach \ya/\yb in {1/2,3/2,3/4,5/4,5/6,7/6}
   {
    \draw [->] (\ya-\xa) -- (\yb-\xa);
    \draw [->] (\yb-\xa) -- (\ya-\xb);
   }
 \foreach \xa/\xb in {0/1,1/2,2/3,3/4,4/5,5/6,6/7,7/8,8/9,9/10,10/11,11/12,12/13,13/14,14/15,15/16}
    \foreach \ya/\yb in {3/3}
   {
     \draw [->] (\ya-\xa)--(a-\xa);
     \draw [->] (a-\xa) --(\ya-\xb);
   }      
 \draw [dashed] (0,0.4) -- (0,7.6); 
 \draw [dashed] (32,0.4) -- (32,7.6);
\end{tikzpicture} } \caption{Indecomposables with a white background may be part of the same \(2\)-cluster tilting object as \(M\) and \(N\).}\label{E8_eks_del1}
\end{figure}
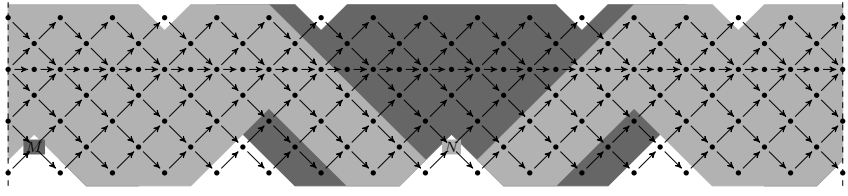

From figure \ref{E8_eks_del1} we find/verify that there can maximally be 4 indecomposables of \(T\) in this \(\tau\)-orbit. If there are 4 indecomposables in this \(\tau\)-orbit then they are closed under \(\tau^4\). This rules out the existence of cases of the form \((?,2+2,?)\). 

Continuing under the assumption that there are 4 indecomposables in this \(\tau\)-orbit that are closed under \(\tau^4\), we have the situation pictured in figure \ref{E8_eks_del2}. From this figure we find that the rest of the summands of \(T\) must be located in \(\tau\)-orbit nr 8. The possible cases at this point are \((2+2,4,0)\) and \((4,4,0)\). Choosing first to focus on the case \((4,4,0)\), we see that this case will occur if one chooses the indecomposables in \(\tau\)-orbit 8 with an arrow going into either \(M,N,O\) or \(R\). The only other possibility is to choose the indecomposables in \(\tau\)-orbit nr. 8 with an arrow from \(M,N,O\) or \(R\). These two choices correspond to the quivers in the top row of table \ref{tableE8}. 

If one instead choose either the indecomposables in \(\tau\)-orbit 8 with an arrow into \(M\) and \(N\), and an arrow from \(O\) and \(R\) one obtains a cluster tilting object with the quiver listed in the second row of table \ref{tableE8}. One may also choose the indecomposables with an arrow from \(M\) and \(N\), and an arrow into \(O\) and \(R\), but this quiver is isomorphic to the first, so there is really only one quiver occuring in this case.

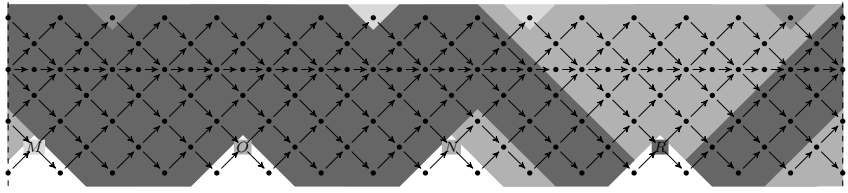
\begin{figure}[h]
\scalebox{.7}{ \begin{tikzpicture}[baseline=-1.7cm,scale=.49,yscale=-1]
\fill[fill1](2,6.5)--(3,7.5)--(5,7.5)--(8,4.5)
--(11,7.5)--(15,7.5)--(17,5.5)--(19,7.5)--(23,7.5)--(26,4.5)--(29,7.5)--(31,7.5)--(32,6.5)--(26,0.5)--
(23,0.5)--(22,1.5)--(21,0.5)--(13,0.5)--(12,1.5)--(11,0.5)
--(8,0.5)--cycle;

\fill[fill12](18,6.5)--(19,7.5)--(21,7.5)--(24,4.5)--(27,7.5)--(31,7.5)--(32,6.5)--(32,0.5)--(29,0.5)--(28,1.5)--(27,0.5)--(24,0.5)--cycle;
\fill[fill12](0,6.5)--(1,5.5)--(3,7.5)--(7,7.5)--(10,4.5)--(13,7.5)--(15,7.5)--(16,6.5)--(10,0.5)--(7,0.5)--(6,1.5)--(5,0.5)--(0,0.5)--cycle;

\fill[fill2](10,6.5)--(11,7.5)--(13,7.5)--(16,4.5)--(19,7.5)--(23,7.5)--(25,5.5)--(27,7.5)--(31,7.5)--(32,7.5)--(32,0.5)--(31,.5)--(30,1.5)--(29,0.5)--(21,0.5)--(20,1.5)--(19,0.5)--(16,0.5)--cycle;
\fill[fill2](0,6.5)--(2,4.5)--(5,7.5)--(7,7.5)--(8,6.5)--(2,0.5)--(0,0.5)--cycle; 

\fill[fill3](26,6.5)--(27,7.5)--(29,7.5)--(32,4.5)--(32,0.5)--cycle;
\fill[fill3](0,4.5)--(3,7.5)--(7,7.5)--(9,5.5)--(11,7.5)--(15,7.5)--(18,4.5)--(21,7.5)--(23,7.5)--(24,6.5)--(18,0.5)--(15,0.5)--(14,1.5)--(13,0.5)--(5,0.5)--(4,1.5)--(3,0.5)--(0,0.5)--cycle;

 \foreach \x in {0,...,16}
  \foreach \y in {1,3,5,7}
   \node (\y-\x) at (\x*2,\y) [vertex] {};
 \foreach \x in {0,...,15}
  \foreach \y in {2,4,6}
   \node (\y-\x) at (\x*2+1,\y) [vertex] {};
 \foreach \x in {0,...,15}
  \foreach \y in {3}
   \node (a-\x) at (\x*2+1,\y) [vertex] {};
   
 \replacevertex{(1,6)}{[tvertex][fill1] {$M$}} 
 \replacevertex{(17,6)}{[tvertex][fill2]{$N$}}
 \replacevertex{(9,6)}{[tvertex][fill2]{$O$}}
 \replacevertex{(25,6)}{[tvertex][fill3]{$R$}}
 \foreach \xa/\xb in {0/1,1/2,2/3,3/4,4/5,5/6,6/7,7/8,8/9,9/10,10/11,11/12,12/13,13/14,14/15,15/16}
  \foreach \ya/\yb in {1/2,3/2,3/4,5/4,5/6,7/6}
   {
    \draw [->] (\ya-\xa) -- (\yb-\xa);
    \draw [->] (\yb-\xa) -- (\ya-\xb);
   }
 \foreach \xa/\xb in {0/1,1/2,2/3,3/4,4/5,5/6,6/7,7/8,8/9,9/10,10/11,11/12,12/13,13/14,14/15,15/16}
    \foreach \ya/\yb in {3/3}
   {
     \draw [->] (\ya-\xa)--(a-\xa);
     \draw [->] (a-\xa) --(\ya-\xb);
   }      
 \draw [dashed] (0,0.4) -- (0,7.6); 
 \draw [dashed] (32,0.4) -- (32,7.6);

\end{tikzpicture} }  \caption{The indecomposable with white background are the ones that may be part of a \(2\)-cluster tilting object togehter with \(M, N, O\) and \(P\).}\label{E8_eks_del2}
\end{figure}

Table \ref{tableE8} lists all quivers of \(2\)-cluster tilting objects of \(\mathcal C_2(E_8)\) such that the corresponding cluster tilting objects is periodic under \(\tau^8\) or \(\tau^4\).

\begin{table}
\centering
\begin{tabular}{|m{2.2cm}|m{1.5cm}|m{9cm}|}
\hline
Distribution in AR-q. & \(\tau^sT=T\) & Corresponding quivers of \(T\) \\ \hline
(4,4,0) & s=4&
 
\scalebox{.8}{\begin{tikzpicture}
\node (a) at (0,0)[vertex]{};
\node (b) at (-1,0)[vertex]{};
\node(c) at (-1,-1)[vertex]{};
\node(d) at (0,-1)[vertex]{};
\node(e) at (0.7,0.7)[vertex]{};
\node(f) at (-1.7,0.7)[vertex]{};
\node(g) at(-1.7,-1.7)[vertex]{};
\node(h) at (0.7,-1.7)[vertex]{};
\node (0) at (0,0.9){};
\node (00) at (0,-1.9){};
\node (000) at (1.3,0){};
\draw[->] (a)--(b);
\draw[->](b)--(c);
\draw[->](c)--(d);
\draw[->](d)--(a);
\draw[->](a)--(e);
\draw[->](b)--(f);
\draw[->](c)--(g);
\draw[->](d)--(h);
\end{tikzpicture}\label{(4,4,0)1}}
 \scalebox{.8}{\begin{tikzpicture}
\node (a) at (0,0)[vertex]{};
\node (b) at (-1,0)[vertex]{};
\node(c) at (-1,-1)[vertex]{};
\node(d) at (0,-1)[vertex]{};
\node(e) at (0.7,0.7)[vertex]{};
\node(f) at (-1.7,0.7)[vertex]{};
\node(g) at(-1.7,-1.7)[vertex]{};
\node(h) at (0.7,-1.7)[vertex]{};

\node (00) at (0,-1.9){};
\node (000) at (1.3,0){};
\draw[->] (a)--(b);
\draw[->](b)--(c);
\draw[->](c)--(d);
\draw[->](d)--(a);
\draw[->](e)--(a);
\draw[->](f)--(b);
\draw[->](g)--(c);
\draw[->](h)--(d);
\end{tikzpicture}\label{(4,4,0)2}}
\\ \hline
(2+2,4,0) &s=8&
\scalebox{.8}{\begin{tikzpicture}
\node (a) at (0,0)[vertex]{};
\node (b) at (-1,0)[vertex]{};
\node(c) at (-1,-1)[vertex]{};
\node(d) at (0,-1)[vertex]{};
\node(e) at (0.7,0.7)[vertex]{};
\node(f) at (-1.7,0.7)[vertex]{};
\node(g) at(-1.7,-1.7)[vertex]{};
\node(h) at (0.7,-1.7)[vertex]{};
\node (0) at (0,0.9){};
\node (00) at (0,-1.9){};
\draw[->] (a)--(b);
\draw[->](b)--(c);
\draw[->](c)--(d);
\draw[->](d)--(a);
\draw[->](a)--(e);
\draw[->](f)--(b);
\draw[->](c)--(g);
\draw[->](h)--(d);
\end{tikzpicture}\label{(2+2,4,0)}} \\ \hline

(2,2,2+2) &s=8&

\scalebox{.75}{\begin{tikzpicture}
\node (a) at (-1,0)[vertex]{};
\node (b) at (0,0)[vertex]{};
\node (c) at (0.7,0.7)[vertex]{};
\node (d) at (1.7,0.7)[vertex]{};
\node (e) at (2.4,0)[vertex]{};
\node (f) at (3.4,0)[vertex]{};
\node (g) at (1.7,-0.7)[vertex]{};
\node (h) at (0.7,-0.7)[vertex]{};
\node (0) at (4,0){};
\node (00) at (0,0.9){};
\node (000) at (0,0.9){};
\node (0000) at (0,-0.3){};
\draw [->](a)--(b);
\draw [->](b)--(c);
\draw[->](c)--(d);
\draw[->](d)--(e);
\draw[->](f)--(e);
\draw[->](e)--(g);
\draw[->](g)--(h);
\draw[->](h)--(b);
\end{tikzpicture}\label{Type(2,2,2+2)1}}
\scalebox{.75}{\begin{tikzpicture}
\node (a) at (-1,0)[vertex]{};
\node (b) at (0,0)[vertex]{};
\node (c) at (0.7,0.7)[vertex]{};
\node (d) at (1.7,0.7)[vertex]{};
\node (e) at (2.4,0)[vertex]{};
\node (f) at (3.4,0)[vertex]{};
\node (g) at (1.7,-0.7)[vertex]{};
\node (h) at (0.7,-0.7)[vertex]{};
\draw [->](b)--(a);
\draw [->](b)--(c);
\draw[->](c)--(d);
\draw[->](d)--(e);
\draw[->](e)--(f);
\draw[->](e)--(g);
\draw[->](g)--(h);
\draw[->](h)--(b);
\end{tikzpicture}\label{Type(2,2,2+2)2}} \\ \hline 
(2+2,2,2) & s=8&
\scalebox{.8}{\begin{tikzpicture}
\node (0) at (0,.3){};
\node(a) at (0,0)[vertex]{};
\node(b) at (1,0)[vertex]{};
\node(c) at(2,0)[vertex]{};
\node(d) at(3,0)[vertex]{};
\node(e) at (0,-1)[vertex]{};
\node(f) at (1,-1)[vertex]{};
\node(g)at (2,-1)[vertex]{};
\node(h) at (3,-1)[vertex]{};
\node (00) at (3.7,0){};
\node (000) at (0,-1.3){};
\draw[->](a)--(b);
\draw[->](b)--(c);
\draw[->](d)--(c);
\draw[->](e)--(f);
\draw[->](g)--(f);
\draw[->](h)--(g);
\draw[->](f)--(b);
\draw[->](c)--(g);
\draw[->](a)--(f);
\draw[->](h)--(c);
\end{tikzpicture}\label{(2+2,2,2)2}}
\scalebox{.8}{\begin{tikzpicture}
\node (0) at (0,.3){};
\node(a) at (0,0)[vertex]{};
\node(b) at (1,0)[vertex]{};
\node(c) at(2,0)[vertex]{};
\node(d) at(3,0)[vertex]{};
\node(e) at (0,-1)[vertex]{};
\node(f) at (1,-1)[vertex]{};
\node(g)at (2,-1)[vertex]{};
\node(h) at (3,-1)[vertex]{};
\node (00) at (0,-1.3){};
\draw[->](a)--(b);
\draw[->](b)--(c);
\draw[->](d)--(c);
\draw[->](e)--(f);
\draw[->](g)--(f);
\draw[->](h)--(g);
\draw[->](f)--(b);
\draw[->](c)--(g);

\draw[->](b)--(e);
\draw[->](g)--(d);
\end{tikzpicture}\label{(2+2,2,2)1}}
\\ \hline

(4,2,2)& s=8&

\scalebox{.8}{\begin{tikzpicture}
\node (0) at (0,.3){};
\node(a) at (0,0)[vertex]{};
\node(b) at (1,0)[vertex]{};
\node(c) at(2,0)[vertex]{};
\node(d) at(3,0)[vertex]{};
\node(e) at (0,-1)[vertex]{};
\node(f) at (1,-1)[vertex]{};
\node(g)at (2,-1)[vertex]{};
\node(h) at (3,-1)[vertex]{};
\node (00) at (3.7,0){};
\node (000) at (0,-1.3){};
\draw[->](b)--(a);
\draw[->](b)--(c);
\draw[->](d)--(c);
\draw[->](e)--(f);
\draw[->](g)--(f);
\draw[->](g)--(h);
\draw[->](f)--(b);
\draw[->](c)--(g);
\draw[->](b)--(e);
\draw[->](g)--(d);
\end{tikzpicture}\label{(4,2,2)1}}
\scalebox{.8}{\begin{tikzpicture}
\node (0) at (0,.3){};
\node(a) at (0,0)[vertex]{};
\node(b) at (1,0)[vertex]{};
\node(c) at(2,0)[vertex]{};
\node(d) at(3,0)[vertex]{};
\node(e) at (0,-1)[vertex]{};
\node(f) at (1,-1)[vertex]{};
\node(g)at (2,-1)[vertex]{};
\node(h) at (3,-1)[vertex]{};
\node (00) at (3.7,0){};
\node (000) at (0,-1.3){};
\draw[->](b)--(a);
\draw[->](b)--(c);
\draw[->](d)--(c);
\draw[->](e)--(f);
\draw[->](g)--(f);
\draw[->](g)--(h);
\draw[->](f)--(b);
\draw[->](c)--(g);
\draw[->](a)--(f);
\draw[->](h)--(c);
\end{tikzpicture}\label{(4,2,2)2}}
\\ \hline
(2+2,0,2+2) & s=8&
\scalebox{.85}{\begin{tikzpicture}
\node (0) at (0,.3){};
\node(a) at (0,0)[vertex]{};
\node(b) at (1,0)[vertex]{};
\node(c) at(2,0)[vertex]{};
\node(d) at(3,0)[vertex]{};
\node(e) at (0,-1)[vertex]{};
\node(f) at (1,-1)[vertex]{};
\node(g)at (2,-1)[vertex]{};
\node(h) at (3,-1)[vertex]{};
\node (00) at (3.7,0){};
\node (000) at (0,-1.3){};
\draw[->](b)--(a);
\draw[->](b)--(c);
\draw[->](c)--(d);
\draw[->](f)--(e);
\draw[->](g)--(f);
\draw[->](g)--(h);
\draw[->](e)--(a);
\draw[->](f)--(b);
\draw[->](c)--(g);
\draw[->](d)--(h);
\draw[->](b)--(e);
\draw[->](g)--(d);
\end{tikzpicture}\label{(2+2,0,2+2)2}}
\scalebox{.85}{\begin{tikzpicture}
\node (0) at (0,.3){};
\node(a) at (0,0)[vertex]{};
\node(b) at (1,0)[vertex]{};
\node(c) at(2,0)[vertex]{};
\node(d) at(3,0)[vertex]{};
\node(e) at (0,-1)[vertex]{};
\node(f) at (1,-1)[vertex]{};
\node(g)at (2,-1)[vertex]{};
\node(h) at (3,-1)[vertex]{};
\node (00) at (3.7,0){};
\node (000) at (0,-1.3){};
\draw[->] (a)--(b);
\draw[->] (b)--(c);
\draw[->](d)--(c);
\draw[->](e)--(f);
\draw[->](g)--(f);
\draw[->](h)--(g);
\draw[->](e)--(a);
\draw[->](f)--(b);
\draw[->](c)--(g);
\draw[->](d)--(h);
\draw[->](b)--(e);
\draw[->](g)--(d);
\end{tikzpicture}\label{(2+2,0,2+2)1}}
\\ \hline
\end{tabular}
\centering\caption{All quivers of \(2\)-cluster tilting objects of \(\mathcal C_2(E_8)\) where the corresponding \(2\)-cluster tilting object is periodic under \(\tau^4\) or \(\tau^8\). The first column gives the distribution of the indecomposable summands of \(T\) in the AR-quiver of \(\mathcal C_2(E_8)\). }
\end{table}\label{tableE8}

\subsection{Summary of results of type \(E_6, E_7\) and \(E_8\).}

Based on the results regardin periodicity of \(2\)-cluster tiltning objects of \(\mathcal C_2(E_6), \mathcal C_2(E_7)\) and \(\mathcal C_2(E_8)\) we obtain the following corollary:

\begin{corollary}
Let \(\mathcal O_F(E_r)=\mathcal D^b(E_r)/F\) for \(r\in\left\{6,7,8\right\}\). Then \(\mathcal O_F(E_r)\) have a \(2\)-cluster tilting object if and only if \(F\) is in the following list:
\begin{itemize}
\item \(F=\tau^{7t}\) for \(E_6\) and \(t\in\mathbb Z\) 
\item \(F=\tau^{7t-1}\left[1\right]\) for \(E_6\) and \(t\in\mathbb Z\)
\item \(F=\tau^{10t}\) for \(E_7\) and \(t\in\mathbb Z\).
\item	\(F=\tau^{4t}\) for \(E_8\) and \(t\in\mathbb Z\).
\end{itemize}
\end{corollary}

\begin{proof}
For the first two parts of the corollary we note that the AR-quiver of \(\mathcal O_F(E_6)\) may be of two main types, cylindrical or of Moebius-shape.
For the third and fourth parts recall that there is only the possibility of having a cylindrical AR-quiver with functors of the form \(\tau^s\left[t\right]\) for orbit categories of type \(\mathcal O_F(E_7)\) and \(\mathcal O_F(E_8)\). 

The results are then obtained by applying lemma \ref{ImCltilt} and lemma \ref{periodInDerived}.
\end{proof}

\section{The Euclidean and wild cases}\label{wildEuclid}
We have in the previous sections \ref{sectA}, \ref{typeD} and \ref{sectE} studied all the cases of representation finite hereditary cases. In this section we will focus on Euclidean and wild algebras, and argue that the phenomenon that we have observed for the representation-finite cases can never occur for cluster tilting objects of representation-infinite algebras.

First we discuss the Euclidean case. The AR-quiver of a path-algbra of some orientation of a Euclidean diagram contains three main components, a preprojective component,  regular components and a preinjective component. In the Euclidean cases the regular component cosists of tubes. 

The following result enables us to consider cluster tilting objects as tilting objects in a suitable module category.

\begin{theorem}\cite{bmrrt}
Let \(H\) be a tame hereditary algebra, and \(T\) a cluster tilting object of \(\mathcal C_2(H)\). Then there is some algebra \(H'\) derived equivalent to \(H\), such that \(T\) is a tilting object in \(\modf H'\).
\end{theorem}

We may now use well-known theory about the placement of the indecomposable summands of tilting objects of path-algebras over a Euclidean diagram. 

\begin{lemma}\cite{AssemSimSkon3}(chapter 17)
Let \(H\) be the path-algebra of a Euclidean diagram, and let \(T\) be a tilting object in \(\modf H\). Then \(T\) has at least two indecomposable non-isomorphic summand that do not lie in regular components. 
\end{lemma}

As a consequence of this lemma, every cluster tilting object \(T\) has an indecomposable summand \(T_p\) lying in the preprojective component or the preinjective component. Since there are infinitely many objects in either component it is impossible to find an integer \(s\) such that \(\tau^sT=T\).

Now let us consider the case of a path algebra over a wild quiver. In this case all the regular components of the AR-quiver are of the type \(\mathbb ZA_\infty\). It follows that in these cases there is no integer \(s\) such that \(\tau^sT=T\) for any \(2\)-cluster tilting object \(T\).


\bibliographystyle{plain}
\bibliography{bibliografi}

\begin{thebibliography}{10}

\bibitem{amiot}
Claire Amiot.
\newblock On the structure of triangulated categories with finitely many
  indecomposables.
\newblock {\em Bull. Soc. Math. France}, 135(3):435--474, 2007.

\bibitem{typeE}
Janine Bastian, Thorsten Holm, and Sefi Ladkani.
\newblock Derived equivalence classification of the cluster-tilted algebras of
  {D}ynkin type {$E$}.
\newblock {\em Algebr. Represent. Theory}, 16(2):527--551, 2013.

\bibitem{bergh}
Petter~Andreas Bergh.
\newblock On the existence of cluster tilting objects in triangulated
  categories.
\newblock {\em J. Algebra}, 417:1--14, 2014.

\bibitem{bow}
M.~A. Bertani-{\O}kland, S.~Oppermann, and A.~Wr{\aa}lsen.
\newblock Finding a cluster-tilting object for a representation finite
  cluster-tilted algebra.
\newblock {\em Colloq. Math.}, 121(2):249--263, 2010.

\bibitem{bmrrt}
Aslak~Bakke Buan, Robert Marsh, Markus Reineke, Idun Reiten, and Gordana
  Todorov.
\newblock Tilting theory and cluster combinatorics.
\newblock {\em Adv. Math.}, 204(2):572--618, 2006.

\bibitem{bmr}
Aslak~Bakke Buan, Robert~J. Marsh, and Idun Reiten.
\newblock Cluster-tilted algebras.
\newblock {\em Trans. Amer. Math. Soc.}, 359(1):323--332 (electronic), 2007.

\bibitem{bpr}
Aslak~Bakke Buan, Yann Palu, and Idun Reiten.
\newblock Algebras of finite representation type arising from maximal rigid
  objects.

\bibitem{HermundD}
Aslak~Bakke Buan and Hermund~Andr{\'e} Torkildsen.
\newblock The number of elements in the mutation class of a quiver of type
  {$D_n$}.
\newblock {\em Electron. J. Combin.}, 16(1):Research Paper 49, 23, 2009.

\bibitem{DagfinnD}
Aslak~Bakke Buan and Dagfinn~F. Vatne.
\newblock Derived equivalence classification for cluster-tilted algebras of
  type {$A_n$}.
\newblock {\em J. Algebra}, 319(7):2723--2738, 2008.

\bibitem{bikr}
Igor Burban, Osamu Iyama, Bernhard Keller, and Idun Reiten.
\newblock Cluster tilting for one-dimensional hypersurface singularities.
\newblock {\em Adv. Math.}, 217(6):2443--2484, 2008.

\bibitem{ccs}
P.~Caldero, F.~Chapoton, and R.~Schiffler.
\newblock Quivers with relations arising from clusters ({$A_n$} case).
\newblock {\em Trans. Amer. Math. Soc.}, 358(3):1347--1364, 2006.

\bibitem{FZ}
Sergey Fomin and Andrei Zelevinsky.
\newblock Cluster algebras. {I}. {F}oundations.
\newblock {\em J. Amer. Math. Soc.}, 15(2):497--529 (electronic), 2002.

\bibitem{Iyama1}
Osamu Iyama.
\newblock Higher dimensional {A}uslander-{R}eiten theory on maximal orthogonal
  subcategories.
\newblock In {\em Proceedings of the 37th {S}ymposium on {R}ing {T}heory and
  {R}epresentation {T}heory}, pages 24--30. Symp. Ring Theory Represent Theory
  Organ. Comm., Osaka, 2005.

\bibitem{Iyama}
Osamu Iyama.
\newblock Higher dimensional {A}uslander-{R}eiten theory on maximal orthogonal
  subcategories.
\newblock In {\em Proceedings of the 37th {S}ymposium on {R}ing {T}heory and
  {R}epresentation {T}heory}, pages 24--30. Symp. Ring Theory Represent Theory
  Organ. Comm., Osaka, 2005.

\bibitem{Iyama2}
Osamu Iyama.
\newblock Higher-dimensional {A}uslander-{R}eiten theory on maximal orthogonal
  subcategories.
\newblock {\em Adv. Math.}, 210(1):22--50, 2007.

\bibitem{keller}
Bernhard Keller.
\newblock On triangulated orbit categories.
\newblock {\em Doc. Math.}, 10:551--581, 2005.

\bibitem{k-zhu}
Steffen Koenig and Bin Zhu.
\newblock {From triangulated categories to abelian categories - cluster tilting
  in a general framework}.
\newblock {\em Mathematische Zeitschrift}, 258(1):143--160, January 2008.

\bibitem{ladkani}
Sefi Ladkani.
\newblock 2-cy-tilted algebras that are not jacobian.

\bibitem{AssemSimSkon3}
Daniel Simson and Andrzej Skowro{\'n}ski.
\newblock {\em Elements of the representation theory of associative algebras.
  {V}ol. 3}, volume~72 of {\em London Mathematical Society Student Texts}.
\newblock Cambridge University Press, Cambridge, 2007.
\newblock Representation-infinite tilted algebras.

\bibitem{HermundA}
Hermund~Andr{\'e} Torkildsen.
\newblock Counting cluster-tilted algebras of type {$A_n$}.
\newblock {\em Int. Electron. J. Algebra}, 4:149--158, 2008.

\bibitem{DagfinnA}
Dagfinn~F. Vatne.
\newblock The mutation class of {$D_n$} quivers.
\newblock {\em Comm. Algebra}, 38(3):1137--1146, 2010.

\bibitem{ZhuXiao}
Jie Xiao and Bin Zhu.
\newblock Relations for the {G}rothendieck groups of triangulated categories.
\newblock {\em J. Algebra}, 257(1):37--50, 2002.

\end{thebibliography}

\end{document}